\documentclass[]{article}
\usepackage{commath} 
\usepackage{amsthm} 
\usepackage{graphicx} 
\usepackage{enumerate} 
\usepackage{xcolor}
\usepackage{algorithm}  
\usepackage{algorithmic}  
\usepackage{bbm}

\usepackage{natbib}

\usepackage{amsmath}
\usepackage{amssymb}
\newtheorem{definition}{Definition}[section]

\newtheorem{theorem}{Theorem}[section]
\newtheorem{proposition}{Proposition}[section]
\newtheorem{remark}{Remark}[section]
\newtheorem{lemma}{Lemma}[section]
\newtheorem{corollary}{Corollary}[section]

\newtheoremstyle{assumption}{6pt}{6pt}{\rm}{}{\sffamily}{ }{ }{}
\theoremstyle{assumption}
\newtheorem{assumption}{\sc Assumption}[section]

\newcommand*{\sprod}[2]{\langle #1, #2 \rangle}
\newcommand*{\bb}[1]{\mathbb{#1}}

\newcommand*{\floor}[1]{{\lfloor #1 \rfloor}}
\newcommand*{\ceil}[1]{{\lceil #1 \rceil}}

\newcommand*{\bt}[1]{\tilde{#1}}

\newcount\colveccount

\def\colvecnext#1{
	#1
	\global\advance\colveccount-1
	\ifnum\colveccount>0
	,
	\expandafter\colvecnext
	\else
	)
	\fi
}

\newtheorem{innercustomgeneric}{\customgenericname}
\providecommand{\customgenericname}{}

\title{Convergence of Preconditioned Hamiltonian Monte Carlo on Hilbert Spaces}

\author{%
	{Jakiw Pidstrigach\thanks{Institute of Mathematics, Universit\"at Potsdam, Karl-Liebknecht-Straße 24-25, 14476 Potsdam, Germany, Email: jakiw.pidstrigach@uni-potsdam.de}}
}

\begin{document}
	
	\maketitle
	
	\begin{abstract}
		{In this article, we consider the preconditioned Hamiltonian Monte Carlo (pHMC) algorithm defined directly on an infinite-dimensional Hilbert space.  In this context, and under a condition reminiscent of strong log-concavity of the target measure, we prove convergence bounds for adjusted pHMC in the standard 1-Wasserstein distance.   
			The arguments rely on a synchronous coupling of two copies of pHMC, which is controlled by adapting elements from Bou-Rabee, Eberle, Zimmer (2020).}  
		
		\textbf{Keywords and phrases:}
		{coupling; convergence to equilibrium; Markov Chain Monte Carlo in infinite dimensions; Hamiltonian Monte Carlo;  Hybrid Monte Carlo; geometric integration; preconditioning; Metropolis-Hastings; Hilbert spaces; Wasserstein distance.}
	\end{abstract}
	
	\section{Introduction}
	\subsection{Hamiltonian Monte Carlo}
	Hamiltonian Monte Carlo (HMC) is a a Markov Chain Monte Carlo (MCMC) method for sampling from complex probability measures whose normalizing constant is unknown. It originated under the name ``Hybrid Monte Carlo'' in \citet{duane1987hybrid} in statistical physics. The `target' measure that HMC can sample from has the form
	\begin{equation}
		\label{equ:hmc_measure_intro}
		\dif \pi(q) \propto \exp(-\Phi(q)) \dif q,
	\end{equation}
	i.e. the target measure has a positive density with respect to the Lebesgue measure $\dif q$. The $\propto$ means that the density of $\pi$ w.r.t. $\dif q$ is proportional to $\exp(-\Phi(q))$ and only differs by a normalizing factor. 
	
	This allows HMC to sample a wide variety of different measures, arising for example in Bayesian statistics and Machine Learning (\citet{bishop2006pattern}), Scientific Computing (\citet{liu2008monte}) and material sciences (\citet{prokhorenko2018large}). A good introduction to Hamiltonian Monte Carlo is given in \citet{neal_2012}.
	
	In each step, HMC proposes a new state by integrating a Hamiltonian system that keeps \eqref{equ:hmc_measure_intro} invariant. Afterwards the new state is either accepted or rejected, based on the energy error made when numerically integrating the Hamiltonian differential equations. We will make this more explicit soon. This results in a Markov chain with invariant measure $\pi$. The following sections give a quick overview of the different flavours of HMC that exist to finally situate this work within the literature.
	
	\subsection{Preconditioned HMC}
	Here we deal with the special case that the measure is $\pi$ is given as
	\begin{equation}
		\label{equ:pi_wrt_to_gaussian_lebesgue}
		\dif\pi(q) \propto \exp(-\Phi(q) - \sprod{q}{C^{-1}q}) \dif q
	\end{equation}
	or equivalently
	\begin{equation*}
		\dif\pi(q) \propto \exp(-\Phi(q)) \dif \pi_0(q)
	\end{equation*}
	where $\pi_0$ is a centred Gaussian measure with covariance matrix $C$. This especially includes the setting where $\pi_0$ is a Gaussian measure on an infinite dimensional Hilbert space. In that case Equation \eqref{equ:pi_wrt_to_gaussian_lebesgue} does not make sense since there is no Lebesgue measure $\dif q$ in infinite dimensions and $\sprod{q}{C^{-1}q}$ would be $\pi_0$-almost surely infinite. 
	The infinite dimensional setting occurs, for example, in Bayesian inverse problems (\citet{stuart2010inverse}),  when sampling paths of conditioned diffusions (\citet{hairer2011signal}), and path integral molecular dynamics (\citet{ChandlerWolynes}).
	Preconditioned HMC was introduced as Hilbert Space Hamiltonian Monte Carlo in \citet{beskos2011hybrid}. Another paper dealing with pHMC and proving some more properties is \citet{ottobre2016function}.
	
	Often the time step size requirement of many MCMC algorithms deteriorates when the dimension $d$ of the sample space increases. HMC already compares favourably to other MCMC algorithms applied to high dimensional perturbations of product measures.  In particular, while the optimal step sizes of the Random Walk Metropolis and Metropolis Adjusted Langevin Algorithm deteriorate like $O(d^{-1})$ and $O(d^{-1/3})$ respectively (\citet{roberts1997weak}, \citet{roberts1998optimal}, \citet{pillai2012optimal}, \citet{roberts2001optimal}), the step size of HMC only deteriorates like $O(d^{-1/4})$ (\citet{beskos2013optimal}). 
	
	Preconditioned HMC exploits the specific structure of $\pi$ given above. As a result, pHMC is a well defined algorithm in an infinite-dimensional Hilbert space. Therefore the stepsize does not deteriorate with the dimension and is $O(1)$. The price to pay is that the limiting Gaussian measure $\pi_0$ has to exist. An important case where high dimensional distributions are of interest, but do not take place on an infinite-dimensional Hilbert space are mean-field models. Recent progress in showing dimension-independence in that case has been made in \citet{bou2020convergence}.

	We will give quantitative bounds on the convergence speed of pHMC to its stationary distribution in a setting analogous to strong convexity of the potential in finite dimensions.

	\subsection{Exact, Adjusted, And Unadjusted HMC}
	When generating the proposal for the next state of the Markov Chain one needs to integrate a Hamiltonian system. The analysis of the algorithm is simplified if one assumes that we can make that integration exact, i.e. with no numerical integration error. This is called \emph{exact HMC}. The acceptance probability for the new proposed state will then always be one. Unfortunately in practice we can only numerically approximate the trajectories of the Hamiltonian system. To remove the integration error bias, we add a Metropolis Hastings Acceptance/Rejection step. This is also called Metropolis-Adjustment of the algorithm. We call this version of the algorithm \emph{adjusted HMC}. This version of the algorithm is actually implementable and is what is often just called ``HMC''. Another interesting version of HMC, recently explored in \citet{bou2020convergence} and \citet{bou2020two} is unadjusted HMC. In this case one does not Metropolis-Adjust the numerical integrator. The dynamics then converge to the ``wrong'' invariant measure $\pi_h$ which depends on the step size in the numerical integrator. One can study the distance of $\pi_h$ to $\pi$. This is also implementable on a computer.
	
	The exact HMC algorithm can be seen as zero-stepsize limit of adjusted HMC or unadjusted HMC and is therefore an important special case. 
	
	\subsection{Convexity of the Potential}
	Another simplification that is often done is to assume that the potential $\Phi$ is strongly convex (i.e. strong log-concavity of the density). This excludes multimodal distributions which are also of great interest in practice but still includes many important special cases.
	
	\subsection{Prior Work}
	Geometric ergodicity of variants of HMC was verified in \citet{bou2017randomized}, \citet{2017arXiv170500166D}, \citet{livingstone2019geometric}. Quantitative bounds on the convergence speed have been proven in the case where $\Phi$ is strongly convex in \citet{mangoubi2017mixing}. In \citet{bou2018coupling} quantitative bounds for exact and adjusted HMC in a general non-convex case are proven. Quantitative convergence bounds for exact pHMC have been shown in \citet{bou2020two} without assuming any convexity of the potential. In \citet{bou2020convergence} explicit, dimension independent convergence rates for unadjusted HMC in the case of mean-field models were shown. In \cite{bou2020couplings} convergence for Anderson Dynamics in high dimensions was studied.
	
	\subsection{Placement of this Work}
	This work closely follows \citet{bou2018coupling} and originates from the author's master thesis under the supervision of Prof. Eberle.
	The work treats the exact and adjusted case of pHMC and assumes something analogous to strong convexity of the potential and a Lipschitz gradient condition. As far as we can tell, it is the first work to obtain explicit bounds on the convergence of pHMC on Hilbert spaces.
	\section{Preliminaries}
	\subsection{Assumptions}
	\label{sec:sobolev_like_spaces}
	This section introduces the basic definitions and assumptions needed to make pHMC work in infinite dimensions. For a more detailed discussion and proofs we refer the reader to \citet{beskos2011hybrid}.
	
	Let $\pi_0 \sim \mathcal{N}(0, C)$ be a centred Gaussian measure on a separable Hilbert space $(\mathcal{H}, \sprod{\cdot}{\cdot})$. Since $C$ is a covariance operator of a Gaussian measure it is trace-class, hence compact, and self-adjoint. Therefore we can apply the spectral theorem to get an $\sprod{\cdot}{\cdot}$-orthonormal eigenbasis $\{\phi_i\}_{i=1}^\infty$ with eigenvalues $\lambda_i^2 \in \mathbb{R}$, i.e.
	\begin{equation*}
		C\phi_i = \lambda_i^2\phi_i.
	\end{equation*}
	We also assume that $C$ is injective, i.e. for all $i$, $\lambda_i > 0$. 
	Using this eigenbasis we can define the \emph{Sobolev-like} spaces $(\mathcal{H}^s, \sprod{\cdot}{\cdot}_s)$ as the closure of $\{\phi_i\}_{i=1}^\infty$ w.r.t. the scalar product
	\begin{equation*}
		\sprod{f}{g}_s = \sum_{j=1}^\infty j^{2s} \sprod{f}{\phi_j} \sprod{g}{\phi_j}.
	\end{equation*}
	We see that $\mathcal{H}^0 = \mathcal{H}$ and for $q < r$, $\mathcal{H}^r \subset \mathcal{H}^q$. Heuristically speaking, higher values of $s$ give us higher regularity and a stricter norm. Introducing these spaces gives us the freedom to prove the bounds in the strongest norm possible.
	We also make the following assumption:
	\begin{assumption}
		\label{ass:std_deviations_decay}
		The standard deviations $\lambda_j > 0$ decay at a polynomial rate $\kappa > \frac{1}{2}$, i.e. 
		\begin{equation*}
			0 < \liminf_{j \to \infty} j^\kappa \lambda_j \le \limsup_{j \to \infty} j^\kappa \lambda_j < \infty.
		\end{equation*}
	\end{assumption}
	\begin{remark}
		\assref{ass:std_deviations_decay} guarantees that $\pi_0$ is supported on any $\mathcal{H}^s$ with $s < \kappa - \frac{1}{2}$ (see Proposition 3.1 of \citet{beskos2011hybrid}).
	\end{remark}
	\begin{remark}
		In the more recent paper \citet{bou2020two}, the spaces $\mathcal{H}^s$ are defined directly in terms of the covariance operator $C$, as $\sprod{x}{y}_s = \sprod{x}{C^{-s}y}$, simplifying \assref{ass:std_deviations_decay}. Our definition and assumption resembles the approach of \citet{beskos2011hybrid}.
	\end{remark}
	Recall that we want to sample a measure that has density proportional to $\exp(-\Phi(q))$ w.r.t. to $\pi_0$. We will also need to make assumptions on the \emph{potential}
	$ \Phi: \mathcal{H} \to \mathbb{R}$.
	\begin{assumption}
		\label{ass:phi_cont_dphi_lipschitz}
		There exists an $l \in \intco{0, \kappa -\frac{1}{2}}$ where $\kappa$ is as in \assref{ass:std_deviations_decay}, such that $\Phi: \mathcal{H}^l \to \bb{R}$ is continuous and the Fréchet derivate $\Dif\Phi: \mathcal{H}^l \to \mathcal{H}^{-l}$ is globally Lipschitz continuous, i.e. there exists an $L > 0$ such that for all $q, q' \in \mathcal{H}^l$
		\begin{equation*}
			\abs{\partial_h \Phi(q) - \partial_h \Phi(q')} \le L \norm{q - q'}_l \norm{h}_l.
		\end{equation*}
		We can also equivalently formulate the above condition through the total derivative as
		\begin{equation*}
			\norm{\Dif\Phi(q) - \Dif\Phi(q')}_{-l} \le L \norm{q - q'}_l.
		\end{equation*}
		
	\end{assumption}
	We now also bound $\Phi$ from below:
	\begin{assumption}
		\label{ass:Phi_lower_growth_bound}
		Fix $l$ as given in Assumption \ref{ass:phi_cont_dphi_lipschitz}. Then for any $\epsilon > 0$ there exists a constant $K = K(\epsilon)$ such that for any $q \in \mathcal{H}^l$
		\begin{equation*}
			\Phi(q) \ge K  - \epsilon \norm{q}_l^2
		\end{equation*}
	\end{assumption}
	This ensures that $\exp(-\Phi(q))$ is integrable w.r.t. $\pi_0$ (due to Fernique's Theorem, see \citet{hairer_spde}) and that the perturbation we make to the Gaussian measure $\pi_0$ is relatively small.
	
	With these assumptions we can state the following Lemma, which is proven as Lemma 4.1 in \citet{beskos2011hybrid}.
	\begin{lemma}
		\label{lemma:inequalities_DPhi}
		There exists a constant $K_{C}$ such that for all $q, q', v, v' \in \mathcal{H}^l$
		\begin{enumerate}[(i)]
			\item $C\Dif\Phi(q) \in \mathcal{H}^l$
			\item $\norm{C\Dif\Phi(q) - C\Dif\Phi(q')}_l \le K_{C}\norm{q - q'}_l$
			\item $|\Dif\Phi(q)(v)| \le K_{C} (1 + \norm{q}_l)\norm{v}_l$
			\item $|\Dif\Phi(q)(v) - \Dif\Phi(q')(v')| \le K_{C}(\norm{v}_l\norm{q - q'}_l + (1 + \norm{q'}_l )\norm{v - v'}_l)$
			\item $\norm{C^{1/2} \Dif \Phi(q)} \le K_{C}(1 + \norm{q}_l)$
			\item $\norm{C^{1/2}\Dif\Phi(q) - C^{1/2}\Dif\Phi(q')} \le K_{C} \norm{q - q'}_l$
		\end{enumerate}
	\end{lemma}
	Item $(i)$ shows us that we can treat $C\Dif\Phi(q)$ as an element of $\mathcal{H}^l$ opposed to $\Dif\Phi(q)$ which is only an element of $\mathcal{H}^{-l}$. The other bounds will also be used often throughout the paper. 
	
	\subsection{Preconditioned Hamiltonian Monte Carlo Algorithm}
	\subsubsection{Preconditioning the Hamiltonian System}
	We will now introduce the algorithm to sample $\pi$ which is given through
	\begin{equation}
		\label{equ:definition_pi_pHMC}
		\dod{\pi}{\pi_0}(q) \propto \exp(-\Phi(q)).
	\end{equation}
	Heuristically, assuming we could write $\pi_0$ as a density w.r.t. a Lebesgue measure as in finite dimensions we would get
	\begin{equation*}
		\dif \pi ``\propto" \exp(-\Phi(q)-\sprod{q}{C^{-1}q})\dif q.
	\end{equation*}
	In the usual derivation of HMC in finite dimensions one would now define the Hamiltonian
	\begin{equation*}
		H(q, p) = \Phi(q) + \frac{1}{2} \sprod{q}{C^{-1}q} + \frac{1}{2} \sprod{p}{M^{-1}p}
	\end{equation*}
	and a measure
	\begin{equation*}
		\dif \Pi \propto \exp(-H(q, p)) \dif q \dif p.
	\end{equation*}
	which has $\pi$ as its $q$-marginal. Here we introduced the auxiliary momentum variable $p$. We view $q$ and $p$ as the position and momentum variables respectively in a physical system with energy $H(q,p)$. A new proposal is generated by first sampling a new momentum $p \sim \mathcal{N}(0, M)$ and then using a numerical integrator (e.g. velocity Verlet) to integrate the Hamiltonian system corresponding the the \emph{Hamiltonian} $H$:
	\begin{equation}
		\label{equ:hamiltonian_dynamics_infinite_dims}
		\begin{cases}
			\dod{q(t)}{t} &= \pd{H}{p} = M^{-1}p(t) \\
			\dod{p(t)}{t} &= -\pd{H}{q} = -C^{-1}q(t) -\Dif \Phi(q(t)),
		\end{cases}
	\end{equation}
	which preserves the energy $H$, therefore $\Pi$, and therefore $\pi$.  
	
	The choice of the mass matrix $M$ is up to the user and can be chosen to make the system easier to solve. In a Hilbert space settings there are two immediate problems with (\ref{equ:hamiltonian_dynamics_infinite_dims}). First of all, since $C$ is the covariance of a Gaussian measure, it is trace class (see \citet{hairer_spde}) and therefore its inverse $C^{-1}$ will be an unbounded operator. Furthermore, $\Dif\Phi(q)$ is an element of $\mathcal{H}^{-l}$ and not necessarily of $\mathcal{H}^{l}$, the space which $q$ and $p$ will be elements of.
	
	In pHMC $M$ is chosen to be $L := C^{-1}$, so that $M^{-1} = C$ cancels out the effect of the unbounded operator $C^{-1}$ and \emph{preconditions} the system. We also work with the velocity variable $v = M^{-1}p = Cp$ instead of the momentum variable $p$. 
	

	The system now becomes
	\begin{equation}
		\label{equ:hamiltonian_dynamics_velocity_infinite_dims}
		\begin{cases}
			\dod{q(t)}{t} &= v(t) \\
			\dod{v(t)}{t} &= -q(t) - C\Dif \Phi(q(t)).
		\end{cases}
	\end{equation}
	We call the flow of this system $\varphi_t$. Note that it cannot be computed explicitly in general due to the appearance of the non-linear term $C\Dif\Phi$.
	Due to \assref{ass:std_deviations_decay} and \assref{ass:phi_cont_dphi_lipschitz} one can verify that $C\Dif\Phi$ is a well defined map from $\mathcal{H}^l$ to $\mathcal{H}^l$ (see $(i)$ of \lemref{lemma:inequalities_DPhi}). Both $q$ and $v$ will be elements of $\mathcal{H}^l$ and there are no unbounded operators appearing in the ODEs. This system now preserves the measure
	\begin{equation}
		\label{equ:Pi}
		\dif \Pi(q, v) \propto \exp(-\Phi(q)) \dif \Pi_0(q, v), \text{ with }\Pi_0(q, v) = \pi_0(q) \otimes \pi_0(v),
	\end{equation}
	which has $\pi$ as its $q$-marginal. 
	
	Formally, using the non-existent Lebesgue-measure in infinite dimensions, the above measure could again be written as
	\begin{equation}
		\label{equ:formal_lebesge_density_infinite_dimensions}
		\dif\Pi(q, v) ``\propto" \exp(-H(q, v)) \dif q \dif v,
	\end{equation}
	with the Hamiltonian
	\begin{equation}
		\label{equ:hamiltonian_infinite_dimensions}
		H(q, v) = \Phi(q) + \frac{1}{2}\sprod{q}{Lq} + \frac{1}{2}\sprod{v}{Lv}.
	\end{equation}
	Since $H(q, v)$ is $\pi_0(q) \otimes \pi_0(v)$-almost surely infinite\footnote{$\sprod{\cdot}{L\cdot}$ is the Cameron-Martin norm which is a.s. infinite, see \citet{hairer_spde}.}. (\ref{equ:Pi}) is the correct definition of $\Pi$, but we will use the formal equations (\ref{equ:formal_lebesge_density_infinite_dimensions}) and (\ref{equ:hamiltonian_infinite_dimensions}) to derive formulas for the acceptance probability.
	
	For implementations of the pHMC algorithm on a computer we need to use a numerical integrator. For that purpose we split (\ref{equ:hamiltonian_dynamics_velocity_infinite_dims}) into two systems,
	\begin{align}
		\label{equ:splitting_hmc_ds_infinitedim2}
		(A) \quad 
		\begin{cases}
			\dod{q(t)}{t} &= v(t) \\
			\dod{p(t)}{t} &= -q(t)
		\end{cases}
	\end{align}
	with flow $\varphi^{(A)}_t$ and
	\begin{align}
		\label{equ:splitting_hmc_ds_infinitedim1}
		(B) \quad 
		\begin{cases}
			\dod{q(t)}{t} &= 0 \\
			\dod{v(t)}{t} &= -C\Dif \Phi(q(t))
		\end{cases}
	\end{align}
	with flow $\varphi^{(B)}_t$. Note that $\varphi^{(A)}$ and $\varphi^{(B)}$ can both be explicitly computed. To approximate the flow of (\ref{equ:hamiltonian_dynamics_velocity_infinite_dims}), we now use the numerical flow
	\begin{equation}
		\label{equ:numerical_integrator_pHMC}
		\psi_t = \varphi^{(B)}_{t/2} \circ \varphi^{(A)}_t \circ \varphi^{(B)}_{t/2}.
	\end{equation}
	The specific form of building $\psi_t$ from $\varphi^{(A)}$ and $\varphi^{(B)}$ in (\ref{equ:numerical_integrator_pHMC}) is called Strang's splitting. For information on how to design numerical integrators for Hamiltonian Monte Carlo, the reader is referred to \citet{bou_rabee_geometric_integration_17} or \citet{leimkuhler_reich_simulation_hd_18}. In finite dimension one splits (\ref{equ:hamiltonian_dynamics_velocity_infinite_dims}) into different systems $(A)$ and $(B)$. One then again uses the same formula (\ref{equ:numerical_integrator_pHMC}) to build $\psi_t$ from $\varphi^{(A)}$ and $\varphi^{(B)}$. The resulting integrator is called velocity Verlet or Leapfrog. In infinite dimensions velocity Verlet would result in a acceptance probability of zero. The integrator we presented here is designed to interact well with infinite dimensional Gaussian measures\footnote{(\ref{equ:splitting_hmc_ds_infinitedim2}) preserves the prior measure $\pi_0(q) \otimes \pi_0(v)$ exactly and (\ref{equ:splitting_hmc_ds_infinitedim1}) has a non-zero acceptance probability since $C\Dif \Phi(q)$ is a element of the Cameron-Martin space of $\pi_0(v)$, see \citet{hairer_spde}.}. 
	
	Instead of making one step of length $t$ with the integrator (\ref{equ:numerical_integrator_pHMC}), one normally makes multiple smaller steps of length $h$. We define
	\begin{equation*}
		\psi^{(T)}_h = \psi_h^{T/h},
	\end{equation*}
	where we assume that the \emph{trajectory length} $T$ is a multiple of the \emph{step size} $h$. $\psi^{(T)}_h$ is the $T/h$-fold concatenation of a step of size $h$.
	\subsubsection{Acceptance Probability}

	The final ingredient for a MCMC algorithm is the acceptance probability. Normally, for HMC it is defined as
	\begin{equation}
		\label{equ:acceptance_probability}
		\alpha(q, v) = \min\{1, \exp(-\Delta H(q, v))\},
	\end{equation}
	where
	\begin{equation*}
		\label{equ:DH}
		\Delta H(q, p) = H(\psi^{(T)}_h(q, v)) - H(q, v).
	\end{equation*}
	is there error in the Hamiltonian we make when using the numerical approximation. If $\psi$ would solve (\ref{equ:hamiltonian_dynamics_velocity_infinite_dims}) exactly, $\Delta H(q, p)$ would be zero. In infinite dimensions, both quantities on the r.h.s. of (\ref{equ:DH}) will be $\pi_0(q) \otimes \pi_0(v)$-almost surely infinite. Therefore we cannot define $\Delta H(q, v)$ in this way. We define it as 
	\begin{equation}
		\label{equ:DeltaH_pHMC}
		\begin{aligned}
			\Delta H(q_0, v_0) =~& \Phi(q_I) - \Phi(q_0) + \frac{h^2}{8}\left(\norm{C^{1/2}\Dif\Phi(q_0)}^2 - \norm{C^{1/2}\Dif\Phi(q_I)}^2\right) \\ - &h\sum_{i=1}^{I-1} \sprod{\Dif \Phi(q_i)}{v_i} - \frac{h}{2}\left(\sprod{\Dif \Phi(q_0)}{v_0} +  \sprod{\Dif \Phi(q_I)}{v_I}\right)
		\end{aligned}
	\end{equation}
	where $(q_i, v_i) = \psi^{i}_h(q_0, v_0)$ and $I = \frac{T}{h}$ which we assume to be a natural number. Note that if we let $h\to0$ formally, we get
	\begin{equation}
		\label{equ:formal_limit_DH}
		\Delta H(q_0, v_0) \xrightarrow[h \to 0]{} \Phi(q_I) - \Phi(q_0) - \int_0^T \Dif\Phi(q_s)v_s \dif s
	\end{equation}
	which is $0$ since physically the gain in potential energy $\Phi(q_I) - \Phi(q_0)$ is equal to the applied force. It can be seen that (\ref{equ:DeltaH_pHMC}) is finite (\S 3.4 of \citet{beskos2011hybrid}).
	
	\begin{remark}[Justification for \ref{equ:DeltaH_pHMC}] 
		\label{rem:justification_DeltaH}
		We will now make a calculation, justifying the formula (\ref{equ:DeltaH_pHMC}). The calculation is rigorous in the finite dimensional case and formal in the infinite dimensional case. We define $(q', v') = \psi_h(q, v)$ as the successors of $(q,v)$ after applying the numerical flow $\psi_h$ (see (\ref{equ:numerical_integrator_pHMC})). Using the definition of $\psi$, $\varphi^{(A)}$, and $\varphi^{(B)}$ we see that
		\begin{align*}
			q' &= \cos(h)q + \sin(h)v  - \frac{h}{2} \sin(h) C \Dif \Phi(q)\\
			v' &= -\sin(h)q + \cos(h)v - \frac{h}{2} \cos(h) C \Dif\Phi(q) - \frac{h}{2} C \Dif \Phi(q').
		\end{align*}
		We have that $H(q', v') = \Psi(q') + \frac{1}{2}\sprod{q'}{Lq'} + \frac{1}{2}\sprod{v'}{Lv'}$. So we start calculating the latter two terms,
		\begin{align*}
			\sprod{q'}{Lq'} =& \cos(h)^2\sprod{q}{Lq} + \sin(h)^2 \sprod{v}{Lv} + \frac{h^2}{4}\sin(h)^2\sprod{\Dif\Psi(q)}{C\Dif\Psi(q)} \\
			+& 2\cos(h)\sin(h)\sprod{q}{Lv} - h\cos(h)\sin(h)\sprod{q}{\Dif\Phi(q)} \\
			-& h\sin(h)^2\sprod{v}{\Dif\Phi(q)},
		\end{align*}
		and
		\begin{align*}
			\sprod{v'}{Lv'} =& \sin(h)^2\sprod{q}{Lq} + \cos(h)^2 \sprod{v}{Lv} + \frac{h^2}{4}\cos(h)^2\sprod{\Dif\Phi(q)}{C\Dif\Phi(q)} \\
			+& \frac{h^2}{4}\sprod{\Dif\Phi(q')}{C\Dif\Phi(q')} - 2\sin(h)\cos(h)\sprod{q}{Lv} \\
			+& h\sin(h)\cos(h)\sprod{q}{\Dif\Phi(q)} + h\sin(h)\sprod{q}{\Dif\Phi(q')} \\
			-& h\cos(h)^2\sprod{v}{\Dif\Phi(q)} - h\cos(h)\sprod{v}{\Dif\Phi(q')} \\
			+& \frac{h^2}{2}\cos(h)\sprod{\Dif\Phi(q)}{C\Dif\Phi(q')}.
		\end{align*}
		Adding these together gives us
		\begin{align*}
			&\sprod{q'}{Lq'} + \sprod{v'}{Lv'} \\
			=~& \sprod{q}{Lq} + \sprod{v}{Lv} + \frac{h^2}{4} (\sprod{\Dif\Phi(q)}{C\Dif\Phi(q)} + \sprod{\Dif\Phi(q')}{C\Dif\Phi(q')})\\
			&- h \sprod{v}{\Dif\Phi(q)} + h\sin(h)\sprod{q}{\Dif\Phi(q')}  - h \cos(h)\sprod{v}{\Dif\Phi(q')} \\
			&+ \frac{h^2}{2}\cos(h)\sprod{\Dif\Phi(q)}{C\Dif\Phi(q')}.
		\end{align*}
		Now we use the formula for $v'$ and a zero-addition of $\frac{h^2}{2}\sprod{C\Dif\Phi(q')}{\Dif\Phi(q')}$ to see that
		\begin{align*}
			&h\sin(h)\sprod{q}{\Dif\Phi(q')}  - h \cos(h)\sprod{v}{\Dif\Phi(q')} 
			+ \frac{h^2}{2}\cos(h)\sprod{\Dif\Phi(q)}{C\Dif\Phi(q')}\\
			=~& -h\sprod{v'}{\Dif\Phi(q')}-\frac{h^2}{2}\sprod{C\Dif\Phi(q')}{\Dif\Phi(q')} ,
		\end{align*}
		and therefore, 
		\begin{align*}
			&\sprod{q'}{Lq'} + \sprod{v'}{Lv'} \\
			=~& \sprod{q}{Lq} + \sprod{v}{Lv} + \frac{h^2}{4} (\sprod{\Dif\Phi(q)}{C\Dif\Phi(q)} - \sprod{\Dif\Phi(q')}{C\Dif\Phi(q')})\\
			&- h (\sprod{v}{\Dif\Phi(q)} + \sprod{v'}{\Dif\Phi(q')}).
		\end{align*}
		We can now use the definition of $H(q,v)$ and $H(q', v')$ to get the final result:
		\begin{align*}
			&H(q', v') - H(q, v)\\
			=~& \Phi(q') - \Phi(q) + \frac{h^2}{8} (\sprod{\Dif\Phi(q)}{C\Dif\Phi(q)} - \sprod{\Dif\Phi(q')}{C\Dif\Phi(q')})\\
			&- \frac{h}{2} (\sprod{v}{\Dif\Phi(q)} + \sprod{v'}{\Dif\Phi(q')}).
		\end{align*}	
		Using this formula iteratively, we get get (\ref{equ:DeltaH_pHMC}). 
		
		This was a formal calculation, since we let the infinities $\sprod{\cdot}{L\cdot}$ cancel each other out. Carrying out the same calculation in finite dimensions is rigorous and one can then see that the energy differences of the finite dimensional approximations converge to $\Delta H$ as defined in (\ref{equ:DeltaH_pHMC}), see Lemma 4.3 of \citet{beskos2011hybrid}.
	\end{remark}

	\begin{algorithm}
		\caption{Exact pHMC Algorithm}
		\label{alg:hs_exact_hmc}
		\begin{algorithmic}[1]
			\REQUIRE $q_0,p_0\in\mathcal{H}^l$, path length $T>0$, number of samples $N$, prior covariance $C$, potential $\Phi$
			\FOR {$n\in\{0,\dots,N-1\}$}
			\STATE Sample $\tilde{v}_n \sim \mathcal{N}(0, C)$
			\STATE Set $(q_{n+1}, v_{n+1}) = \varphi_T(q_n, \tilde{v}_n)$ with $\varphi_T$ defined through (\ref{equ:hamiltonian_dynamics_velocity_infinite_dims})
			\ENDFOR
			\RETURN Samples $(q_{i})_{i=0}^N$
		\end{algorithmic}
	\end{algorithm}
	\begin{algorithm}
		\caption{Adjusted pHMC Algorithm}
		\label{alg:hs_hmc}
		\begin{algorithmic}[1]
			\REQUIRE $q_0,p_0\in\mathcal{H}^l$, path length $T>0$, step size $h>0$, number of samples $N$, prior covariance $C$, potential $\Phi$
			\FOR {$n\in\{0,\dots,N-1\}$}
			\STATE Sample $\tilde{v}_n \sim \mathcal{N}(0, C)$
			\STATE Set proposal state $(q_{n+1}^*, v_{n+1}^*) = \psi^{(T)}_h(q_n, \tilde{v}_n)$ with $\psi^{(T)}_h = \psi^{\lfloor T/h \rfloor}_h$ and $\psi_h$ defined through (\ref{equ:splitting_hmc_ds_infinitedim1}), (\ref{equ:splitting_hmc_ds_infinitedim2}), (\ref{equ:numerical_integrator_pHMC})
			\STATE Calculate $\alpha(q_n, v_n) = 1\wedge \exp(-\Delta H(q_n, v_n)) \in \intcc{0, 1}$ with $\Delta H$ as in (\ref{equ:DeltaH_pHMC})
			\STATE Set $(q_{n+1}, v_{n+1}) = \gamma(q_n^*, v_n^*) + (1 - \gamma)(q_{n+1}, v_{n+1})$ where $\gamma \sim \mbox{Bernoulli}(\alpha(q_n, v_n))$
			\ENDFOR
			\RETURN Samples $(q_{i})_{i=0}^N$
		\end{algorithmic}
	\end{algorithm}
	\subsection{Statement of the Algorithms}
	
	We are now ready to state the exact and adjusted pHMC algorithms, Algorithm \ref{alg:hs_exact_hmc} and Algorithm \ref{alg:hs_hmc} respectively. 
	
	Since $\varphi$ exactly preserves $H$ there is no need for an acceptance-rejection step in exact pHMC, since the acceptance probability will always be one.
	
	\section{Convergence of pHMC}
	\label{sec:conv_pHMC}
	In this section we will treat the convergence of pHMC to its stationary distribution $\pi$. We will use a coupling technique to get an explicit rate of the convergence of pHMC to its stationary distribution. These results were proven for standard finite-dimensional HMC in \citet{bou2018coupling} and are now also proven for exact pHMC in \citet{bou2020two} without assuming convexity of the potential.
	
	We start by discretizing the space to get a finite dimensional space $\mathcal{H}_N$ in \secref{sec:proofs_pHMC}. We then prove the results from \citet{bou2018coupling} for the integrator (\ref{equ:numerical_integrator_pHMC}) instead of velocity Verlet. The results will not depend on the embedding dimension and we therefore can take the limit in the end. The results in \citet{bou2018coupling} do depend explicitly on the dimension because the sampled velocity is i.i.d. Gaussian. Expectations including the velocity therefore depend on the dimension. In our case we will sample $v^N$ such that expectations including $v^N$ can be upper bounded by expectations of $v \sim \mathcal{N}(0, C)$, which are dimensionless. We will start in \secref{sec:fdd_approx} by introducing a spectral approximation technique which we will use later on. In \secref{sec:ass_conv} we will state the additional assumptions we need to make to get the convergence results. With these we can then state the precise results and the coupling technique in \secref{sec:main_results_contraction}. In \secref{sec:proofs_contraction_pHMC} we then prove the results.

	\subsection{Finite Dimensional Approximations.}
	\label{sec:proofs_pHMC}
	\label{sec:fdd_approx}
	We discretize $\mathcal{H}$ to an $N$-dimensional space $\mathcal{H}_N$ using a spectral technique and will prove our results in finite dimensions. The infinite dimensional case will be recovered as a limit.
	
	We define 
	\begin{equation*}
		\mathcal{H}_N = \left\{q \in \mathcal{H}: q = \sum_{j=1}^N q_j \phi_j, ~ q_j \in \bb{R}\right\},
	\end{equation*}
	where $\phi_j$ are the eigenfunctions of $C$ as in \secref{sec:sobolev_like_spaces}. We denote by $\mbox{proj}_{\mathcal{H}_N}$ the projection of $\mathcal{H}$ onto $\mathcal{H}_N$. We will also denote the projection of $\mathcal{H}^l$ for $l \not= 0$ onto $\mathcal{H}_N$ by $\mbox{proj}_{\mathcal{H}_N}$ in cases where this does not cause confusion. We will treat $\mathcal{H}_N$ as $\bb{R}^N$ where the basis we choose for the isomorphism is $\{\phi_i\}_{i=1}^N$, i.e. when we speak of an operator in terms of a matrix, then we mean the matrix representation with respect to $\{\phi_i\}_{i=1}^N$. We denote by $q^N$ the projection of $q$ to $\mathcal{H}_N$ and similarly for $v^N$, i.e.
	\begin{equation*}
		q^N = \mbox{proj}_{\mathcal{H}_N} q, ~ v^N = \mbox{proj}_{\mathcal{H}_N} v.
	\end{equation*}
	We see that the restriction of $C$ to $\mathcal{H}_N$, $C_N$ takes the simple form
	\begin{equation*}
		C_N = \mbox{diag}\{\lambda_1^2, \ldots, \lambda_N^2\}.
	\end{equation*}
	Using that we define the distributions $\pi_{0, N}$ and $\pi_N$ on $\mathcal{H}_N \approx \bb{R}^N$ as
	\begin{equation*}
		\pi_{0, N} = \mathcal{N}(0, C_N), \qquad \dif\pi_N \propto \exp(-\Phi_N(q))\dif\pi_{0, N}(q).
	\end{equation*}
	$\pi_{0, N}$ is equal to the image measure of $\pi_0$ under $\mbox{proj}_{\mathcal{H}_N}$. $\Phi_N$ is defined as the restriction of $\Phi$ to $\mathcal{H}_N$. When viewing $\Dif\Phi$ as an element of the dual space ${\mathcal{H}^l}^* = \mathcal{H}^{-l}$ then $D\Phi_N = \mbox{proj}_{\mathcal{H}_N}D\Phi$. We then consider the system
	\begin{equation}
		\label{equ:fdd_dynamics_pHMC}
		\begin{cases}
			\dod{q}{t} &= v\\
			\dod{v}{t} &= -q - C_N\Dif\Phi_N(q) = -q - C\mbox{proj}_{\mathcal{H}_N}\Dif\Phi(q)
		\end{cases},
	\end{equation}
	and denote its flow by $\varphi_{t, N}$. We define $\psi_{h,N}$ analogous to $\psi_h$, but replace every occurrence of $C$ by $C_N$ and every occurrence of $\Dif\Phi$ by $\Dif\Phi_N = \mbox{proj}_{\mathcal{H}_N}\Dif\Phi$. Since $C_N$ is just the restriction of $C$ to $\mathcal{H}_N$ and $C$ keeps $\mathcal{H}_N$ invariant, we will often just use $C$ instead of $C_N$ when applying it to elements of $\mathcal{H}_N$. For $q, v \in \mathcal{H}_N$, the dynamics (\ref{equ:fdd_dynamics_pHMC}) are associated to the Hamiltonian 
	\begin{equation*}
		H_N(q, v) = \Phi_N(q) + \frac{1}{2}\sprod{q}{C_N^{-1}q} + \frac{1}{2}\sprod{v}{C_N^{-1}v},
	\end{equation*}
	which leads to the acceptance probability
	\begin{equation}
		\alpha_N(q, v) = \min\{1, \exp(-\Delta H_N(q, v))\}
	\end{equation}
	where 
	\begin{equation*}
		\Delta H_N(q, v) = H_N(\psi^{(T)}_{h, N}(q, v)) - H_N(q, v),
	\end{equation*} 
	The acceptance probability can also be rewritten as (\ref{equ:DeltaH_pHMC}) as we have seen in \remref{rem:justification_DeltaH}.
	The following propositions now establishes some bounds for the discrete dynamics that are independent of $N$ and convergence estimates such that we can use limiting arguments later on.
	\begin{proposition}
		\label{prop:exact_dynamics_fd_converge}
		The following holds.
		\begin{enumerate}[(i)]
			\item For any $T>0$ there exists a unique solution of (\ref{equ:fdd_dynamics_pHMC}) in the space $C^1([-T, T], \mathcal{H}_N \times \mathcal{H}_N)$.
			\item Let $\varphi_{T, N}$ denote the group flow of(\ref{equ:fdd_dynamics_pHMC}). Then $\varphi_{T, N}$ is globally Lipschitz with respect to the norm induced by $\mathcal{H}^l \times \mathcal{H}^l$ with a Lipschitz constant of the form $\exp(KT)$, where $K$ is independent of $N$ and only depends on $C$ and $\Phi$.
			\item For each $T > 0$ there exists a $C(T) > 0$ independent of $N$ such that for $0 \le t \le T$ and $q^0, v^0 \in \mathcal{H}^l$, if we set
			\begin{equation*}
				({q^N(t)},{v^N(t)}) = \varphi_{t, N}(\mbox{\normalfont{proj}}_{\mathcal{H}_N}(q^0), \mbox{\normalfont{proj}}_{\mathcal{H}_N}(v^0)),
			\end{equation*}
			then
			\begin{align*}
				\norm{q^N(t)}_l + \norm{v^N(t)}_l &\le C(T)(1 + \norm{q^N(0)}_l + \norm{v^N(0)}_l) \\
				&\le C(T)(1 + \norm{q(0)}_l + \norm{v(0)}_l).
			\end{align*}
			Moreover,
			\begin{equation}
				\label{equ:conv_exact_dynamics}
				\begin{aligned}
					\norm{q^N(t) - q(t)}_l + \norm{v^N(t) - v(t)}_l \xrightarrow[N \to \infty]{} 0
				\end{aligned}
			\end{equation}
			and, for any $s \in (l, \kappa - \frac{1}{2})$,
			\begin{equation}
				\label{equ:conv_bound_exact_dynamics}
				\begin{aligned}
					&\norm{q^N(t) - q(t)}_l + \norm{v^N(t) - v(t)}_l \\
					\le& C(T)\left(\frac{1}{N^{s- l}}(\norm{q^0}_s + \norm{v^0}_s) + \frac{1}{N}(1 + \norm{q^0}_l + \norm{v^0}_l)\right),
				\end{aligned}
			\end{equation}
			where $(q(t), v(t)) = \varphi_t(q^0, v^0)$.
		\end{enumerate}
	\end{proposition}
	We now state a similar proposition for the numerical integrator.
	\begin{proposition}
		\label{prop:discrete_dynamics_inequalities}
		\begin{enumerate}[(i)]
			\item $\psi^{(T)}_{h, N}$ is globally Lipschitz map on $\mathcal{H}^l_N \times \mathcal{H}_N^l$ with a Lipschitz constant of the form $\exp(KT)$, where $K$ is independent of $N$ and depends only on $\pi_0$ and $\Phi$.
			\item For each $T > 0$ there exists a $C(T) > 0$ independent of $N$ such that for $0 \le i \le \lfloor T/h \rfloor$ and $q^0, v^0 \in \mathcal{H}^l$, if we set
			\begin{equation}
				\label{equ:definition_qv_Ni}
				(q^{N, i}, v^{N, i}) = \psi^i_{h, N}(\mbox{\normalfont{proj}}_{\mathcal{H}_N}q^0, \mbox{\normalfont{proj}}_{\mathcal{H}_N}v^0),
			\end{equation}
			then
			\begin{equation}
				\label{equ:growth_bound_discretized}
				\begin{aligned}
					\norm{q^{N, i}}_l + \norm{v^{N, i}}_l &\le C(T)\left(1 + \norm{q^{N, 0}}_l + \norm{v^{N, 0}}_l\right) \\
					&\le C(T)(1 + \norm{q^0}_l + \norm{v^0}_l)
				\end{aligned}.
			\end{equation}
			Moreover
			\begin{equation}
				\label{equ:fdd_qv_vanish}
				\begin{aligned}
					&\norm{q^{N, i} - q^i}_l + \norm{v^{N, i} - v^i}_l \xrightarrow[N \to \infty]{} 0
				\end{aligned}
			\end{equation}		
			and, for any $s \in \intoo{l, \kappa - \frac{1}{2}}$,
			\begin{equation}
				\label{equ:fdd_qv_converge}
				\begin{aligned}
					&\norm{q^{N, i} - q^i}_l + \norm{v^{N, i} - v^i}_l \\
					\le& C(T)\left(\frac{1}{N^{s-l}} (\norm{q^0}_s + \norm{v^0}_s) + \frac{1}{N} (1 + \norm{q^0}_l + \norm{v^0}_l) \right)
				\end{aligned}
			\end{equation}
		\end{enumerate}
	\end{proposition}
	Lastly we will need that the acceptance probabilities also converge to their infinite dimensional counterpart:
	\begin{lemma}
		\label{lemma:convergence_acceptance_prob}
		As $N \to \infty$
		\begin{equation*}
			\alpha_N(q^N, v^N) \to \alpha(q, v)
		\end{equation*}
		for every $(q, v)\in \mathcal{H}^l \times \mathcal{H}^l$.
	\end{lemma}
	All three of these Propositions are proven in \citet{beskos2011hybrid}.
	
	\subsection{Assumptions and Notation}
	\label{sec:ass_conv}
	\subsubsection{Regularity of $\Phi$ and Integration Length}
	The first assumption we make is very similar to \assref{ass:phi_cont_dphi_lipschitz} and actually just a convenient reformulation. 
	\begin{assumption}
		\label{ass:DPhiBoundL}
		In \assref{ass:phi_cont_dphi_lipschitz} we defined $L$ such that
		\begin{equation*}
			\norm{\Dif\Phi(q) - \Dif\Phi(q')}_{-l} \le L \norm{q - q'}_l.
		\end{equation*}
		We still assume that the above holds but now we assume that $L$ is also big enough to fulfil
		\begin{equation}
			\label{equ:DPhiDifferencebound}
			\norm{C\Dif\Phi(q) - C\Dif\Phi(q')}_l \le L\norm{q - q'}_l,
		\end{equation}
		i.e.
		\begin{equation*}
			\sprod{C\Dif\Phi(q)}{h}_l \le L\norm{q - q'}_l\norm{h}_l.
		\end{equation*}
		We define $L'$ such that
		\begin{equation*}
			\norm{\Dif\Phi(q)}_{-l}\le \norm{\Dif\Phi(q) - \Dif\Phi(0)}_{-l} + \norm{\Dif\Phi(0)}_{-l} \le L\norm{q}_l + L'
		\end{equation*}
		and again assume that the same holds for $\norm{C\Dif\Phi(q)}_l$, i.e. 
		\begin{align}
			\label{equ:DPhiGrowthBound}
			\norm{C\Dif\Phi(q)}_l \le L\norm{q}_l + L'
		\end{align}
	\end{assumption}
	\begin{remark}
		By \lemref{lemma:inequalities_DPhi} we know that \assref{ass:phi_cont_dphi_lipschitz} implies (\ref{equ:DPhiDifferencebound}). But the constant in general could differ from $L$ by a factor. So all that we did is assume that $L$ is large enough to fulfil both requirements. 
	\end{remark}
	We also assume something reminiscent of $\sprod{q}{Lq} + \Phi(q)$ being a convex function. But since we are using the preconditioned dynamics the statement is not exactly equivalent to any convexity assumption.
	
	
	\begin{assumption}
		\label{ass:PhiConvex}
		There is a $\zeta > 0$ such that for all $x, y \in \mathcal{H}^l$
		\label{ass:convexity}
		\begin{equation*}
			\sprod{x - y}{x-y}_l + \sprod{C\Dif\Phi(x) - C\Dif\Phi(y)}{x - y}_l \ge \zeta\norm{x - y}_l^2 
		\end{equation*}
		and $\zeta$, $L$ are such that $\zeta \le L + 1$.
	\end{assumption}
	This assumption can for example be fulfilled when $\Dif\Phi$ is a Lipschitz map with sufficiently small Lipschitz constant. 
	\begin{remark}
		For $\zeta < 1$ a sufficient condition for \assref{ass:PhiConvex} is that 
		\begin{equation*}
			\norm{C\Dif\Phi(x) - C\Dif\Phi(y)}_l \le (1 - \zeta)\norm{x-y}_l.
		\end{equation*}
		Therefore $L \le 1 - \zeta$ in \assref{ass:DPhiBoundL} would directly fulfil both conditions of \assref{ass:PhiConvex}.
	\end{remark}
	We also need to make some assumptions on the second derivative of $\Phi$. We therefore need to define the operator norm.
	\begin{definition}[Operator Norm]
		Let $(V, \norm{\cdot}_V)$ and $(W, \norm{\cdot}_W)$ be two normed spaces.
		We define the operator norm of a map linear $F: V \to W$ as
		\begin{equation*}
			\norm{F}_{V \to W} = \sup_{v \in V} \frac{\norm{Fv}_W}{\norm{v}_V}.
		\end{equation*}
	\end{definition}
	We can now state the assumption on the second derivative.
	\begin{assumption}
		\label{ass:DPhi2BoundM}
		We assume that $\Phi: \mathcal{H}^l \to \bb{R}$ is in $C^4$ and all derivatives are bounded:
		
		The second derivative evaluated at $x \in \mathcal{H}^l$, $\Dif^2\Phi(x)$, is identified with a linear map from $\mathcal{H}^s$ to $\mathcal{H}^{-s}$. We assume that
		\begin{equation}
			\label{equ:lipsch_statement_for_D2}
			\norm{D^2\Phi(x) - D^2\Phi(y)}_{l \to -l} \le M\norm{x - y}_l,
		\end{equation}
		i.e.
		\begin{equation*}
			\norm{\Dif^3\Phi(q)} \le M
		\end{equation*}
		or
		\begin{equation*}
			\norm{\partial_{h_1} \partial_{h_2}\Phi(x) - \partial_{h_1} \partial_{h_2}\Phi(y)}_{l \to -l} \le M\norm{x - y}_l\norm{h_1}_l\norm{h_2}_l.
		\end{equation*}
		We again assume that the same constant $M$ also fulfils
		\begin{equation}
			\label{equ:lipsch_statement_for_CD2}
			\norm{CD^2\Phi(x) - CD^2\Phi(y)}_{l \to l} \le M\norm{x - y}_l.
		\end{equation}
		Finally we assume that
		\begin{equation*}
			\norm{\Dif^4\Phi(q)} \le N.
		\end{equation*}
	\end{assumption}
	\begin{remark}
		Again by \lemref{lemma:inequalities_DPhi} we know that a statement similar to (\ref{equ:lipsch_statement_for_CD2}) is implied by (\ref{equ:lipsch_statement_for_D2}). Therefore all we assume by (\ref{equ:lipsch_statement_for_CD2}) is that the constant is the same.
	\end{remark}
	\begin{remark}
		Bounds that now depend explicitly on $L, M$ or $L'$ depend implicitly on $C$ due to \assref{ass:DPhiBoundL} and \assref{ass:DPhi2BoundM}. This does not cause any problems since in our setting the bounds that can depend on $L, M$ or $L'$ also are allowed to depend on $C$. For the proofs we will discretize the space in the following and work on $\bb{R}^N$. Therefore a subtlety we have to watch out for is that the bounds do not depend on the dimension $N$ of the discretized space; since neither $L, M, L'$ nor $C$ do, we cannot introduce dimension dependence through the backdoor by depending on them.
	\end{remark}
	The next two assumptions are a bit more technical. We will run the dynamics for time $T$. We need to make a restriction on that time and also on the step size $h$ of the discrete dynamics.
	\begin{assumption}
		\label{ass:Ltht1}
		We assume that
		\begin{equation}
			\label{equ:LTh_1_cond}
			(T^2 + L(T^2 + 2hT)) \le 1.
		\end{equation}	
	\end{assumption}
	
	\begin{assumption}
		\label{ass:LthtzetaL}
		We assume that
		\begin{equation}
			\label{equ:LTh_1_cond_gamma_L}
			(T^2 + L(T^2 + 2hT)) \le \frac{\zeta}{1 + L}.
		\end{equation}	
	\end{assumption}
	\begin{remark}
		Together with \assref{ass:PhiConvex}, \assref{ass:LthtzetaL} implies \assref{ass:Ltht1}, since we assume that $\zeta \le L + 1$.
	\end{remark}
	
	\subsubsection{Continuous Interpolation}
	
	
	We denote $(q^k, v^k) := \psi_h^k(q^0, v^0)$.
	If we fully write out one step of pHMC we get
	\begin{align*}
		q^{k+1} &= \cos(h)q^k + \sin(h)v^k  - \frac{h}{2} \sin(h) C \Dif \Phi(q^k)\\
		v^{k+1} &= -\sin(h)q^k + \cos(h)v^k - \frac{h}{2} \cos(h) C \Dif\Phi(q^k) - \frac{h}{2} C \Dif \Phi(q^{k+1})
	\end{align*}
	(see \remref{rem:justification_DeltaH}), where $h \in \intcc{0, t}$ is the stepsize. In the following lemmas we will interpolate between the samples to get a continuous path. That will allow us to  use arguments involving differential equations. The first thing that comes into mind is a linear interpolation, i.e.
	\begin{equation}
		\begin{aligned}
			q_t &= q_{\lfloor t \rfloor} + \frac{\bt{t}}{h}\left(({\cos(h) - 1}) q_{\lfloor t \rfloor } + \sin(h)v_{\lfloor t \rfloor} - \frac{h}{2}\sin(h)C\Dif\Phi(q_{\lfloor t \rfloor})\right) \\
			v_t &= v_{\lfloor t \rfloor} \\
			&+\frac{\bt{t}}{h}\left(
			-\sin(h)q_{\lfloor t \rfloor} + (\cos(h) - 1)v_{\lfloor t \rfloor} - \frac{h}{2} \cos(h)C \Dif\Phi(q_{\lfloor t \rfloor}) - \frac{h}{2} C \Dif \Phi(q_{\lceil t \rceil})
			\right)
		\end{aligned}
	\end{equation}
	where $\bt{t} = t - \floor{t}$, $\lfloor t \rfloor = \max\{sh | s \in \bb{Z}, sh \le t\}$, and $\lceil t \rceil = \min\{sh | s \in \bb{Z}, sh \ge t\}$. Note that $q_{kh} = q^k$ and similarly for $v$. This leads to the differential equations
	\begin{equation}
		\label{equ:ode_discr_dynamics_linear}
		\begin{cases}
			\dod{q_t}{t} &= \frac{1}{h}\left((\cos(h) - 1) q_{\lfloor t \rfloor } + \sin(h)v_{\lfloor t \rfloor} - \frac{h}{2}\sin(h)C\Dif\Phi(q_{\lfloor t \rfloor})\right) \\
			\dod{v_t}{t} &= \frac{1}{h}\left(-\sin(h)q_{\lfloor t \rfloor} + (\cos(h)-1)v_{\lfloor t \rfloor} - \frac{h}{2} \cos(h)C \Dif\Phi(q_{\lfloor t \rfloor}) - \frac{h}{2} C \Dif \Phi(q_{\lceil t \rceil})\right).
		\end{cases}
	\end{equation}
	Due to the linearity, the right hand side of these differential equations is constant on $\intcc{\floor{t}, \ceil{t}}$ and therefore easy to bound. 
	Another way to interpolate is to use the stepsize $h$ as time variable, i.e.
	\begin{equation}
		\label{equ:ode_discr_interpolation_nonlinear}
		\begin{aligned}
			q_t &= \cos(\bt{t}) q_{\lfloor t \rfloor } + \sin(\bt{t})v_{\lfloor t \rfloor} - \frac{\bt{t}}{2}\sin(\bt{t})C\Dif\Phi(q_{\lfloor t \rfloor})\\
			v_t &= -\sin(\bt{t})q_{\lfloor t \rfloor} + \cos(\bt{t})v_{\lfloor t \rfloor} 
			- \frac{\bt{t}}{2} \cos(\bt{t}) C \Dif\Phi(q_{\lfloor t \rfloor}) - \frac{\bt{t}}{2} C \Dif \Phi(q_{\lceil t \rceil}),
		\end{aligned}
	\end{equation}
	leading to the differential equations
	\begin{equation}
		\label{equ:ode_discr_dynamics_nonlinear}
		\begin{cases}
			\dod{q_t}{t} &= 
			v_t 
			- \frac{\sin(\bt{t})}{2} C\Dif\Phi(q_{\floor{t}})
			+ \frac{\bt{t}}{2}C \Dif \Phi(q_{\ceil{t}}) \\
			\dod{v_t}{t} &= -q_t - \frac{\cos(\bt{t})}{2}C\Dif\Phi(q_{\lfloor t \rfloor}) - \frac{1}{2}C\Dif\Phi(q_\ceil{t}).
		\end{cases}
	\end{equation}
	The differential equation resemble closely the exact equations (\ref{equ:hamiltonian_dynamics_velocity_infinite_dims}). Another important difference is that $\od{v_t}{t}$ does not depend on $v_t$, so that there is the possibility to eliminate $v_t$ and write it as a second order equation. That is not possible for the differential equations which we obtain from the linear interpolation. We will use the second, non-linear interpolation instead of the linear one. \citet{bou2018coupling} uses the linear interpolation to make the same estimates for the velocity Verlet integrator. In their case however, the linear interpolation fulfils a much nicer equation than (\ref{equ:ode_discr_dynamics_linear}) and one can also eliminate $v$ in the linear case. 
	
	We will allow the stepsize of $h$ to be $0$. When we choose $h=0$ we are using the exact dynamics $\varphi_t$ instead of the approximation $\psi^{(t)}_h$ together with a continuous interpolation. Note that for $h=0$ the differential equations fulfilled by the interpolation, (\ref{equ:ode_discr_dynamics_nonlinear}), indeed reduce to the exact dynamics since $\bt{t} = t - \floor{t} = 0$. 
	
	\subsection{Generalized Reversibility of adjusted pHMC}
	Here we prove that the adjusted pHMC algorithm on $\mathcal{H}^l$ satisfies generalized reversibility. For this purpose, we introduce the velocity sign flip
	\begin{equation*}
		S(q, v) = (q, -v).
	\end{equation*}
	We use $x$ and $y$ to denote elements of $\mathcal{H}^l \times \mathcal{H}^l$, i.e. $x$ and $y$ each consist of the position and velocity coordinate.
	The transition kernel of adjusted pHMC on $\mathcal{H}^l \times \mathcal{H}^l$ is given by
	\begin{equation*}
		p(x, \dif y) = \delta_{\psi(x)}(\dif y)\alpha(x) + \delta_{x}(\dif y) \, (1 - \alpha(x)),
	\end{equation*}
	where $\psi$ could either stand for the exact flow $\varphi$ or the numerical approximation $\psi^{(t)}_h$. We introduce the velocity flipped transition 
	\begin{equation*}
		p^S(x, \dif y) = p(S(x), S(\dif y)) = \delta_{\psi \circ S(x) }(S(\dif y))\alpha(S(x)) + \delta_{S(x)}(\dif y)(1 - \alpha(S(x))).
	\end{equation*}
	We say that pHMC fulfils the generalized \emph{generalized reversibility} condition with respect to $\Pi$ and $S$ if $(p\Pi)(x, y) = (p^S\Pi)(y, x)$, i.e.
	\begin{equation}
		\label{equ:generalized_reversibility}
		\Pi(\dif \, (q, v))p((q, v), \dif \, (q', v')) = \Pi(\dif \, (q', v'))p((q', -v'), \dif \, (q, -v)).
	\end{equation}
	The Markov Chain produced by pHMC has this property:
	\begin{proposition}[Generalized Reversibility]
		Let $p$ be the transition kernel of adjusted pHMC. Then it fulfils the generalized reversibility condition with respect to $\Pi$ and $S$
	\end{proposition}
	\begin{proof}
		We want to show that $(p\Pi)(A, B) = (p^S\Pi)(B, A)$ for Borel sets $A, B \in \mathcal{B}(\mathcal{H}^l \times \mathcal{H}^l)$. It is enough to show that for any continuous, bounded function $k$ on $(\mathcal{H}^l \times\mathcal{H}^l) \times (\mathcal{H}^l \times\mathcal{H}^l)$, we have that
		\begin{equation*}
			\mathbb{E}_{(x, y) \sim p\Pi}[k(x, y)] = \mathbb{E}_{(x, y) \sim p^S\Pi}[k(y, x)].
		\end{equation*}
		It therefore suffices to show that
		\begin{align*}
			&\int_{\mathcal{H}^{l, 4}} \left(k(x, \psi(x)) \alpha(x) + k(x, x) (1-\alpha(x))\right) e^{-\Phi(x)} \dif \Pi_0(x) \\
			=& \int_{\mathcal{H}^{l, 4}} \left(k(S\circ\psi\circ S(x), x) \alpha(S(x)) + k(S(x), x) (1-\alpha(S(x)))\right) e^{-\Phi(x)} \dif \Pi_0(x).
		\end{align*}
		There is no normalization constant appearing in the above equations since we can just rescale $k$ by the inverse of the normalization constant. Since finite dimensional HMC fulfils the above equation we have that 
		\begin{align*}
			&\int_{\mathcal{H}^{l, 4}} \left(k(x_N, \psi_N(x_N)) \alpha_N(x_N) + k(x_N, x_N) (1-\alpha_N(x_N))\right) e^{-\Phi(x_N)} \dif \Pi_0(x) \\
			=& \int_{\mathcal{H}^{l, 4}} k(S\circ\psi\circ S(x_N), x_N) \alpha(S(x_N))e^{-\Phi(x_N)} \\
			& \qquad + k(S(x_N), x_N) (1-\alpha(S(x_N))) e^{-\Phi(x_N)} \dif \Pi_0(x).
		\end{align*}
		Here we used that $\Pi_{N, 0} = (\text{proj}_{\mathcal{H}_N \times \mathcal{H}_N})_* \Pi_0$ to get rid of the $N$-dependence in the measure itself. Taking the limit $N\to\infty$, we have that $x_N \to x$ and $\alpha_N(x_N) \to \alpha(x)$ due to \propref{prop:exact_dynamics_fd_converge}, \propref{prop:discrete_dynamics_inequalities} and \lemref{lemma:convergence_acceptance_prob}. Therefore the integrands converge pointwise to their infinite dimensional counterparts. The term $e^{-\Phi}$ can be upper bounded by $e^{K}e^{-\epsilon \norm{\cdot}_l}$ by \assref{ass:Phi_lower_growth_bound} for any $\epsilon$ and by choosing $\epsilon$ appropriately this is integrable due to Fernique's theorem (see \citet{hairer_spde}). Since $k$ and $\alpha_N$ are also bounded we can use dominated convergence to see that the above integrals converge to $\mathbb{E}_{(x, y) \sim p\Pi}[k(x, y)]$ and $\mathbb{E}_{(x, y) \sim p^S\Pi}[k(y, x)]$ respectively. Therefore these are equal and we conclude the proof.
	\end{proof}

	Generalized reversibility implies that the chain is invariant with respect to $\Pi$ if we move to $S(x)$ in case of a rejection. Note that this means that we would need to flip the velocity sign in Algorithm \ref{alg:hs_hmc} and Algorithm \ref{alg:hs_exact_hmc} in case of a rejection. But since we do sample a new velocity in the next iteration, this would not make a difference in practice. However, if one would do only partial velocity refreshments one could not skip this step.
	
	\subsection{Coupling of the Velocities}
	\label{sec:main_results_contraction}
	
	The main results of this section are contraction estimates for pHMC. One step of exact pHMC consists of sampling a new velocity and then letting the state evolve using Hamiltonian dynamics. The next state is given as
	\begin{equation*}
		q'(q, v) = \varphi_T(q, v),
	\end{equation*}
	where $\varphi$ is the exact flow of the Hamiltonian system (\ref{equ:hamiltonian_dynamics_velocity_infinite_dims}).
	
	We now couple two copies of exact pHMC at two different positions $x, y \in \mathcal{H}^l$ by sampling only one velocity $v \sim \mathcal{N}(0, C)$ and using that same velocity for both transitions. We want to show that on average, where the expectation is taken over $v$, these dynamics move closer to each other. The final result will be the following theorem.

	\begin{theorem}
		\label{thm:exact_pHMC_contractive_hs}
		Assume that \assref{ass:DPhiBoundL}, \assref{ass:PhiConvex} and \assref{ass:DPhi2BoundM} hold. Assume that $T$ is such that for $h=0$, \assref{ass:LthtzetaL} holds, i.e. $T^2(L + 1) \le \frac{\zeta}{L + 1}$. Then for any $x, y \in \mathcal{H}^l$,
		\begin{equation*}
			\bb{E}\left[\norm{x'(x, v) - y'(y, v)}_l\right] \le \left(1 - \frac{1}{27}\zeta T^2\right) \norm{x - y}_l.
		\end{equation*}
	\end{theorem}


	We also derive an analogous result for the case where we numerically approximate $\varphi$ by $\psi$, see (\ref{equ:numerical_integrator_pHMC}). For given $q\in \mathcal{H}^l$ and a fixed $v \in \mathcal{H}^l, U \in \intcc{0, 1}$ the transition map of adjusted pHMC takes the form
	\begin{equation}
		\label{equ:q'}
		q'(q, v, U) = \psi^{(T)}_h(q, v) \mathbbm{1}_{\{U \le \alpha(q, v)\}} + (q, v)\mathbbm{1}_{\{U > \alpha(q, v)\}}.
	\end{equation}
	We again couple two copies of adjusted pHMC at two different positions $x, y \in \mathcal{H}^l$. We do this by not only using the same $v$, but also using the same $U$ to determine $x'(x, v, U)$ and $y'(y, v, U)$ by (\ref{equ:q'}). The coupling of the velocities together with the convexity of the potential will lead to a decrease in $\norm{\psi_h^{(T)}(x, v) - \psi_h^{(T)}(y, v)}_l$. The coupling of $U$ will make the probability that one of the moves is rejected and the other one is accepted small.
	
	\begin{theorem}
		\label{thm:contraction_hs}
		Assume that \assref{ass:DPhiBoundL}, \assref{ass:PhiConvex} and \assref{ass:DPhi2BoundM} hold. Assume that $T$ and $h_1$ are such that \assref{ass:LthtzetaL} is satisfied for $h = h_1$. Then there exists an $h_0$ such that for any $0 < h \le \min\{h_0, h_1\}$ with $T/h \in \bb{Z}$ and for any $x, y \in \mathcal{H}^l$ with $\max\{\norm{x}_l, \norm{y}_l\} \le R$ we have
		\begin{equation*}
			\bb{E}\left[\norm{x'(x, v, U) - y'(y, v, U)}_l\right] \le \left(1 - \frac{1}{27}\zeta T^2\right) \norm{x - y}_l
		\end{equation*}
		where $R$ is any real number such that $\bb{P}[\norm{v}_l > R] \le \frac{\zeta}{2000(L+1)}$. The expectation is taken over $v$ and $U$. Furthermore for fixed $L, L', M$ and $C$ the dependence of $h_0^{-1}$ on $T$ and $R$ is of the form $\mathcal{O}((1 + T^{-1/2})(1 + R^2))$.
	\end{theorem}
	
	From \thmref{thm:exact_pHMC_contractive_hs} we can now conclude:
	\begin{corollary}
		\label{cor:wasserstein_distance_decay}
		Let $\mathcal{W}^1$ be the $L^1$ Wasserstein distance and $P$ is the transition kernel of exact pHMC. Then, for any two distributions $\mu$ and $\nu$ on $\mathcal{H}^l$ it holds that
		\begin{equation*}
			\mathcal{W}^1(\mu P^n, \nu P^n) \le \left(1 - \frac{1}{27}\zeta T^2\right)^n \mathcal{W}^1(\mu, \nu).
		\end{equation*}
	\end{corollary}
	Setting $\mu$ to the stationary distribution of $P$ we see that the $\mathcal{W}^1$ distance decreases exponentially and we have an explicit expression for the rate. 
	
	For the adjusted case, \thmref{thm:contraction_hs}, we do not directly get a bound on the Wasserstein distance since we have a requirement on the size of $\norm{x}_l$ and $\norm{y}_l$. But an appropriate choice of $R$ still gives us a good understanding of how fast the trajectories converge.

	\subsection{Proofs}
	\label{sec:proofs_contraction_pHMC}
	\subsubsection{Notation in Finite Dimensions}
	We first prove the above theorems hold for the discretized pHMC. We discretize the algorithm and the space as we did in \secref{sec:proofs_pHMC}, i.e. by replacing $\alpha$ with $\alpha_N$, $C$ with $C_N$, $\Phi$ with $\Phi_N$, and the Hamiltonian dynamics with their finite dimensional counterpart (\ref{equ:fdd_dynamics_pHMC}) in Algorithm \ref{alg:hs_exact_hmc}
	and Algorithm \ref{alg:hs_hmc}. We will see that all bounds only depend on $L, M, \zeta$ and $C$ which are all quantities that are independent of the discretization dimension $N$. In the end we take the limit $N \to \infty$ to conclude that the results also hold on the Hilbert space itself.	
	
	Since the following propositions will be quite technical already, we will denote $q, v \in \bb{R}^N$ and drop the notational dependence on $N$. Only in the end, when we take the limit, we will again write $q^N$ and $v^N$ for the projections of $q, v\in\mathcal{H}^l$ to $\mathcal{H}_N$. We will also just use $\Phi$ and $C$ instead of $\Phi_N$ and $C_N$. As long as we make the upper bounds in terms of $L, M, \zeta$ and $C$, that does not matter since these constants are dimensionless: $\Dif\Phi_N(q) = \mbox{proj}_{\mathcal{H}_N}\Dif\Phi(q)$ fulfils the same bounds with the same constants $L$ and $M$ on $\mathcal{H}_N$ as $\Dif\Phi$ does on $\mathcal{H}^l$. Also expectations of $v^N \sim \mathcal{N}(0, C_N)$ can in our case always be upper bounded by the corresponding expectations of $v \sim \mathcal{N}(0, C)$.
	
	\subsubsection{Proofs in Finite Dimensions}
	We now proceed to the proofs. The first two lemmas are some preliminary estimates that give us growth bounds for the dynamics itself and for the difference of the dynamics. In these proofs we use that we can eliminate $v$ and get a second order equation. The following \lemref{lemma:contraction} then states that the dynamics converge to each other when the velocity is the same. This lemma, together with \lemref{lemma:DH} which then makes statements about the energy error and thus the acceptance probability are the main technical workhorses of this section. Afterwards we proceed with proving the main theorems. 
	
	We from now on denote by $q_t(x, v)$ and $v_t(x, v)$ the position and velocity coordinate at time $t$ when started in $x, v$. 
	\begin{lemma}
		\label{lemma:qv_prioribounds}
		Assume that \assref{ass:Ltht1} holds.
		Let $x, v \in \bb{R}^N$ be arbitrary. Let $q_t(x, v), v_t(x, v)$ fulfil (\ref{equ:ode_discr_dynamics_linear}) with initial conditions $x, v$. Then if $t \in h\bb{Z}$
		\begin{align*}
			\max_{s \le t}\norm{q_s(x, v) - (x + sv)}_l \le& (t^2 + L(t^2 + 2th))\max(\norm{x}_l, \norm{x + tv}_l) \\
			&+ L'(t^2 + 2th)
		\end{align*}
		and
		\begin{align*}
			\max_{s \le t}\norm{v_s(x, v) - v}_l \le& (1+L)t \max_{s \le t}\norm{q_s(x, v)}_l + L't \\
			\le& (1 + L)t(1 + t^2 + L(t^2 + 2th))\max(\norm{x}_l, \norm{x + tv}_l)\\
			&+ L'((1+ L)(t^3 + 2t^2h) + t).
		\end{align*}
		Especially
		\begin{align*}
			\max_{s \le t}\norm{q_s(x, v)}_l \le& 2\max\{\norm{x}_l, \norm{x + tv}_l\} + L'(t^2 + 2th) \le K(1 + \norm{x}_l + \norm{v}_l)
		\end{align*}
		and
		\begin{align*}
			\max_{s \le t}\norm{v_s(x, v)}_l \le& \norm{v}_l + (1+L)t(2\max\{\norm{x}_l, \norm{x + tv}_l\} + L'(t^2 + 2th)) +L't \\
			=& \norm{v}_l + 2(1+L)t\max\{\norm{x}_l, \norm{x + tv}_l\} \\
			&+  L't((1+ L)(t^2 + 2th) + 1) \\
			\le& K(1 + \norm{x}_l + \norm{v}_l),
		\end{align*}
		where $K$ is a constant depending only on $L, L'$.
	\end{lemma}
	\begin{proof}
		We assume $t \in h\bb{Z}$ and therefore $t = \floor{t} = \ceil{t}$. We use the nonlinear interpolation resulting in dynamics (\ref{equ:ode_discr_dynamics_nonlinear}). Let $q_t = q_t(x, v)$ and $v_s = v_s(x, v)$. Also let $\bt{t} = t - \floor{t}$ and similarly for $r, u$. Finally define $q_t^* = \max_{s \le t}\norm{q_s}_l$. The maximum is well defined since $q_s$ is continuous and $\intcc{0, t}$ is compact. Now we calculate
		\begin{align*}
			q_t - q_0 &= \int_0^t v_r \dif r
			+ \int_0^t - \frac{\sin(\bt{r})}{2} C\Dif\Phi(q_{\floor{r}})
			+ \frac{\bt{r}}{2}C \Dif \Phi(q_{\ceil{r}}) \dif r
		\end{align*}
		and eliminate $v$ by calculating
		\begin{align*}
			\int_0^t v_r \dif r= tv_0 
			+ \int_0^t \int_0^r -q_u - \frac{\cos(\bt{u})}{2}C\Dif\Phi(q_{\floor{u}}) - \frac{1}{2}C\Dif\Phi(q_{\ceil{u}}) \dif u \dif r.
		\end{align*}
		We now see that
		\begin{align*}
			&\norm{q_t - q_0 - tv_0}_l \le \frac{1}{2}t^2(q_t^*(1 + L) + L') + th(Lq_t^* + L') \\
			=& \frac{1}{2} q_t^* (t^2 + L(t^2 + 2th)) +  \frac{1}{2}L'(t^2 + 2th)\\
			\le& \frac{1}{2}\left(\max_{s \le t}{\norm{q_s - q_0 - sv_0}}_l + \max_{s \le t}\norm{q_0 + sv_0}_l\right)(t^2 + L(t^2 + 2th))) \\
			&+ \frac{1}{2}L'(t^2 + 2th),
		\end{align*}
		where we used that $\sin(\tilde{r}), \cos(\tilde{u}) \le 1$ and (\ref{equ:DPhiGrowthBound}) in \assref{ass:DPhiBoundL}. Now we use \assref{ass:Ltht1} to conclude that
		\begin{equation*}
			\max_{s \le t}\norm{q_t - q_0 - tv_0}_l \le (t^2 + L(t^2 + 2th))\max\left\{\norm{q_0}_l, \norm{q_0 + tv_0}_l\right\} + L'(t^2 + 2th)
		\end{equation*}
		and
		\begin{align*}
			q_t^* \le& (1 + (t^2 + L(t^2 + 2th)))\max\left\{\norm{q_0}_l, \norm{q_0 + tv_0}_l\right\} + L'(t^2 + 2th) \\
			\le& 2 \max\left\{\norm{q_0}_l, \norm{q_0 + tv_0}_l\right\} + L'(t^2 + 2th)
		\end{align*}
		where we used \assref{ass:Ltht1}. 
		Now we calculate for $v_s$ that
		\begin{align*}
			\norm{v_s - v_0}_l \le&(1 + L)tq_t^* + L't\\
			\le& (1+L)t(1 + t^2 + L(t^2 + 2th))\max\left\{\norm{q_0}_l, \norm{q_0 + tv_0}_l\right\} \\
			&+L'((1+ L)(t^3 + 2t^2h) + t)
		\end{align*}
		and
		\begin{align*}
			\norm{v_s}_l \le& \norm{v_0}_l + 2t(1+L)\max\left\{\norm{q_0}_l, \norm{q_0 + tv_0}_l\right\} \\
			&+ L'((1+ L)(t^3 + 2t^2h) + t).
		\end{align*}
		To get the last two inequalities, first upper bound $h$ by $t$ and then note that by \assref{ass:Ltht1} we can upper bound $t$ by a constant depending only on $L$.
	\end{proof}
	The proof of the next lemma works similarly. But instead of (\ref{equ:DPhiGrowthBound}) we can directly use (\ref{equ:DPhiDifferencebound}) since we are treating differences all the time. That spares us the $L'$ terms.
	We will often deal with the term 
	\begin{equation}
		\label{equ:As}
		A_s := \Dif\Phi(q_s(x, v)) - \Dif\Phi(q_s(y, u)),
	\end{equation}
	
	and therefore introduce the shorthand notation $A_s$.
	\begin{lemma}
		\label{lemma:zw_prioribounds}
		Assume that \assref{ass:Ltht1} holds.
		Let $x,y, u, v \in \bb{R}^N$ be arbitrary. Let $q_t(x, v), v_t(x, v)$ fulfil to (\ref{equ:ode_discr_dynamics_linear}) with initial conditions $x, v$. Then if $t \in h\bb{Z}$
		\begin{align*}
			&\max_{s \le t} \norm{q_s(x, v) - q_s(y, u) - (x - y) - s(v - u)}_l \\
			\le& (t^2 + L(t^2 + 2th))\max\left\{\norm{x - y}_l, \norm{x - y + t(u - v)}_l\right\}
		\end{align*}
		and
		\begin{equation}
			\label{equ:priori_bound_ws}
			\begin{aligned}
				&\max_{s \le t} \norm{v_s(x, v) - v_s(y, u) - (v - u)}_l\\
				\le& t(1+L)(1 + t^2 + L(t^2 + 2th))\max\left\{\norm{x - y}_l, \norm{x - y + t(v - u)}_l\right\}
			\end{aligned}
		\end{equation}
	\end{lemma}
	\begin{proof}
		The proof is similar to the one of the last lemma. Let $z_t = q_t(x, v) - q_t(y, u)$, $w_t = v_t(x, v) - v_t(y, u)$, $z_t^* = \max_{s \le t}\norm{z_s}_l$, and $w_t^* = \max_{s \le t}\norm{w_s}_l$. Observe that $z_t$ and $w_t$ fulfil differential equations similar to the ones for $q$ and $v$:
		\begin{equation}
			\label{equ:dynamics_zt_wt}
			\begin{cases}
				\dod{z_t}{t} &= 
				w_t 
				- \frac{\sin(\bt{t})}{2} CA_{\lfloor t \rfloor}
				+ \frac{\bt{t}}{2}C A_{\lceil t \rceil} \\
				\dod{w_t}{t} &= -z_t - \frac{\cos(\bt{t})}{2}CA_{\lfloor t \rfloor} - \frac{1}{2}CA_\ceil{t},
			\end{cases}
		\end{equation}
		where $A_s$ is defined in (\ref{equ:As}).
		Now we proceed similar as before and therefore need to calculate
		\begin{align*}
			\int_0^t w_s \dif s =& 
			tw_0 + \int_0^t \int_0^r - z_u - \frac{\cos(\bt{u})}{2}CA_{\lfloor u \rfloor} - \frac{1}{2}CA_\ceil{u} \dif u \dif r.
		\end{align*}
		Therefore
		\begin{align*}
			\norm{z_t - z_0 - tw_0}_l \le \frac{1}{2}t^2 (1 + L)z_t^* + thLz_t^* = \frac{1}{2}z_t^*(t^2 + L(t^2 + 2th))
		\end{align*}
		and
		\begin{align*}
			\max_{s \le t} \norm{z_s - z_0 - rw_0}_l  \le (t^2 + L(t^2 + 2th))\max\left\{\norm{z_0}_l, \norm{z_0 + tw_0}_l\right\}.
		\end{align*}
		For $w_s$ we have
		\begin{align*}
			\max_{s \le t}\norm{w_s - w_0}_l \le& t(1+L)z_t^* \\
			\le& t(1+L)(1 + t^2 + L(t^2 + 2th))\max\left\{\norm{z_0}_l, \norm{z_0 + tw_0}_l\right\}
		\end{align*}
		which concludes the proof.
	\end{proof}
	We are especially interested in the case where $v = u$ since this is how we couple the dynamics. For that case the above inequalities imply
	\begin{align}
		\label{equ:same_vel_zbound}
		&\max_{s \le t} \norm{q_s(x, v) - q_s(y, v) - (x - y)}_l 
		\le (t^2 + L(t^2 + 2th))\norm{x - y}_l
	\end{align}
	and
	\begin{align}
		\label{equ:same_vel_wbound}
		&\max_{s \le t} \norm{v_s(x, v) - v_s(y, v)}_l
		\le 2t(1+L)\norm{x - y}_l
	\end{align}
	if \assref{ass:Ltht1} holds. 
	Now we come to the first bigger result of this section: the numerical approximation of the Hamiltonian dynamics are indeed contractive on small time intervals if they have the same initial velocity.
	\begin{lemma}
		\label{lemma:contraction}
		Suppose that \assref{ass:convexity}, \assref{ass:DPhiBoundL}, \assref{ass:DPhi2BoundM} and \assref{ass:LthtzetaL} hold. Then
		\begin{equation*}
			\norm{q_t(x, v) - q_t(y, v)}_l^2 \le (1 - \frac{1}{12}\zeta t^2)\norm{x - y}_l^2
		\end{equation*}
		for any $x, y \in \mathcal{H}^l$ such that
		\begin{equation}
			\label{equ:ass_size_bound_x_v}
			(1 + \norm{x}_l + \norm{v}_l)h \le \frac{\zeta}{K}
		\end{equation}
		where $K \in \intoo{0, \infty}$ is a constant depending only on $C, L, L'$ and $M$. 
	\end{lemma}
	\begin{proof}
		We define $a(t) = \norm{z_t}_l^2$ and $b(t) = 2\sprod{z_t}{w_t}_l$, where $z_t$ and $w_t$ are defined in (\ref{equ:dynamics_zt_wt}). We see
		\begin{align*}
			\dod{a(t)}{t} =& 2\sprod{z_t}{w_t}_l 
			-{\sin(\bt{t})}\sprod{z_t}{CA_{\lfloor t \rfloor}}_l + \tilde{t}\sprod{z_t}{C A_{\lceil t \rceil}}_l \\
			=& b(t) + \delta(t) \\
			\dod{b(t)}{t} =& -2\sprod{z_t}{z_t}_l  -{\cos(\bt{t})}\sprod{z_t}{CA_{\lfloor t \rfloor}}_l - \sprod{z_t}{CA_\ceil{t}}_l \\
			&+ 2\sprod{w_t}{w_t}_l -{\sin(\bt{t})}\sprod{w_t}{CA_{\lfloor t \rfloor}}_l + \tilde{t}\sprod{w_t}{C A_{\lceil t \rceil}}_l \\
			=& -\gamma(t) + \eta(t) + 2\norm{w_t}_l^2 + \epsilon(t)
		\end{align*}
		where
		\begin{align*}
			\delta(t) =& -{\sin(\bt{t})}\sprod{z_t}{CA_{\lfloor t \rfloor}}_l + \tilde{t}\sprod{z_t}{C A_{\lceil t \rceil}}_l \\
			\gamma(t) =& 2\sprod{z_t}{z_t}_l + 2\sprod{z_t}{CA_t}_l\\
			\eta(t) = & -{\cos(\bt{t})}\sprod{z_t}{CA_{\lfloor t \rfloor}}_l - \sprod{z_t}{CA_\ceil{t}}_l + 2\sprod{z_t}{CA_t}_l \\
			\epsilon(t) =&-{\sin(\bt{t})}\sprod{w_t}{CA_{\lfloor t \rfloor}}_l + \tilde{t}\sprod{w_t}{C A_{\lceil t \rceil}}_l
		\end{align*}
		and $A_s$ is defined in (\ref{equ:As}).
		By \assref{ass:PhiConvex} it holds that $\gamma(t) \ge 2\zeta a(t)$. Therefore,
		\begin{equation*}
			\dod{b(t)}{t} = -2\zeta a(t) + \beta(t)
		\end{equation*}
		with
		\begin{equation}
			\label{equ:definition_beta_proof}
			\beta(t) \le \eta(t) + 2\norm{w_t}_l^2 + \epsilon(t).
		\end{equation}
		This leads us to the initial value problem
		\begin{equation*}
			\begin{cases}
				\dod{a(t)}{t} &= b(t) + \delta(t) \\
				\dod{b(t)}{t} &= -2\zeta a(t) + \beta(t)
			\end{cases}
		\end{equation*}
		with initial conditions $a(0) = \norm{z_0}_l^2$ and $b(0) = 0$. It has the unique solution 
		\begin{equation}
			\label{equ:solution_ode_a}
			\begin{aligned}
				a(t) =& \cos\left(\sqrt{2\zeta} t\right)\norm{z_0}_l^2 + \int_0^t \cos\left(\sqrt{2\zeta}(t - r)\right)\delta(r)\dif r \\
				&- \frac{1}{\sqrt{2\zeta}} \sin\left(\sqrt{2\zeta}t\right)\delta(0) + \int_0^t \frac{1}{\sqrt{2\zeta}} \sin\left(\sqrt{2\zeta}(t - r)\right)\beta(r)\dif r
			\end{aligned}
		\end{equation}
		For $h=0$, the case where we use the exact dynamics, $\delta$ and $\beta$ vanish. For $h > 0$ we will now bound $\delta, \epsilon$ and $\beta$. We start by bounding $-{\sin(\bt{t})} CA_{\lfloor t \rfloor}
		+ {\bt{t}}C A_{\lceil t \rceil}$ to bound $\delta(t)$ and $\epsilon(t)$. Note that
		\begin{align*}
			&\bt{t}C \Dif \Phi(z_\ceil{t})-\sin(\bt{t}) CA_\floor{t} \\
			=& \bt{t}(C A_{\lceil t \rceil}-CA_{\lfloor t \rfloor}) + (\bt{t} - \sin(\bt{t})) CA_{\lfloor t \rfloor}
		\end{align*}
		For the first term we denote $q_s(x, v)$ by $x_s$ and $q_s(y, v)$ by $y_s$ and calculate
		\begin{equation}
			\label{equ:calculation_DPhi_borderterms}
			\begin{aligned}
				&\norm{C A_{\lceil t \rceil}-CA_{\lfloor t \rfloor}}_l \\
				=& \norm{C\Dif\Phi(x_\ceil{t}) - C\Dif\Phi(x_\floor{t}) - (C\Dif\Phi(y_\ceil{t}) - C\Dif\Phi(y_\floor{t}))}_l\\
				=& \norm{\int_\floor{t}^\ceil{t} C\Dif^2\Phi(x_r)\dot{x_r} - C\Dif^2\Phi(y_r)\dot{y_r} \dif r}_l \\
				=& \norm{\int_\floor{t}^\ceil{t} (C\Dif^2\Phi(x_r) - C\Dif^2\Phi(y_r))\dot{x_r} + C\Dif^2\Phi(y_r)\dot{z_r} \dif r}_l \\
				\le& Mhz_\ceil{t}^*\left(v_\ceil{t}^* + h\left(Lx_\ceil{t}^* + L'\right)\right) + Lh\left(w_\ceil{t}^* + hLz_\ceil{t}^*\right)
			\end{aligned}
		\end{equation}	
		where we used \assref{ass:DPhiBoundL} and \assref{ass:DPhi2BoundM}. 
		
		For the second term note that we have
		\begin{equation}
			\label{equ:x_sin_difference_bound}
			\abs{h - \sin(h)} = h - \sin(h) \le \frac{1}{6}h^3
		\end{equation}
		since the first two derivatives at $0$ of $h - \sin(h)$ and $\frac{1}{6}h^3$ coincide at $0$ whereas the third derivatives are $\cos(h)$ and $1$ respectively and $\cos(h) \le 1$. 
		With that we can bound $\delta(t)$ by
		\begin{align*}
			\norm{\delta(t)}_l \le& z_t^*\left(h\norm{C A_\ceil{t}-CA_\floor{t}}_l + \frac{1}{6}h^3\norm{CA_\floor{t}}_l\right)\\
			\le& h^2z_\ceil{t}^*\left(Mz_\ceil{t}^*\left(v_\ceil{t}^* + h\left(Lx_\ceil{t}^* + L'\right)\right) + L\left(w_\ceil{t}^* + hLz_\ceil{t}^*\right)\right) \\
			&+ \frac{1}{6} h^3 L {z_\ceil{t}^*}^2 \\
			\le& h^2 {z_\ceil{t}^*}^2 \left(M\left(v_\ceil{t}^* + h\left(Lx_\ceil{t}^* + L'\right)\right) + L\left(2t(L+1) + hL\right)\right) \\
			&+ \frac{1}{6}h^2{z_\ceil{t}^*}^2\left(hL\right),
		\end{align*}
		where we used that $w_t^* \le 2t(L+1)z_t^*$ by \lemref{lemma:zw_prioribounds} and especially (\ref{equ:same_vel_wbound}).
		Similarly we bound $\epsilon(t)$ by
		\begin{align*}
			\norm{\epsilon(t)}_l \le& w_t^*\left(h\norm{C A_\ceil{t}-CA_\floor{t}}_l + \frac{1}{6}h^3\norm{CA_\floor{t}}_l\right)\\
			\le& h^2w_\ceil{t}^*\left(Mz_\ceil{t}^*\left(v_\ceil{t}^* + h\left(Lx_\ceil{t}^* + L'\right)\right) + L\left(w_\ceil{t}^* + hLz_\ceil{t}^*\right)\right) \\
			&+ \frac{1}{6} h^3 L w^*_\ceil{t}{z_\ceil{t}^*}\\
			\le& 2t(L+1)h^2{z_\ceil{t}^*}^2\left(M\left(v_\ceil{t}^* + h\left(Lx_\ceil{t}^* + L'\right)\right) + L\left(2t(L+1) + hL\right)\right) \\
			&+ \frac{1}{3}t(L+1)h^2{z_\ceil{t}^*}^2 \left( h L \right),
		\end{align*}
		which is just the bound of $\norm{\delta(t)}_l$ multiplied by $2t(L+1)$.
		Now the first integral in (\ref{equ:solution_ode_a}) can be bounded by
		\begin{align*}
			&\int_0^t \cos(\sqrt{2\zeta}(t-r)) \delta(r) \dif r \le t \delta^*(t),
		\end{align*}
		where $\delta^*(t) = \max_{s \le t} \abs{\delta(t)}$.
		To bound the second integral in (\ref{equ:solution_ode_a}) we need to bound $\beta$ and therefore $\eta$. We calculate
		\begin{align*}
			&\norm{\cos(\bt{t})CA_\floor{t} + CA_\ceil{t} - 2CA_t}_l \\
			\le&\abs{\cos(\bt{t}) - 1}\norm{CA_\floor{t}}_l + \norm{CA_\floor{t} + CA_\ceil{t} - 2CA_t}_l \\
			\le& \frac{1}{2}h^2Lz_\ceil{t}^*+ Mhz_\ceil{t}^*\left(v_\ceil{t}^* + h\left(Lx_\ceil{t}^* + L'\right)\right) + Lh\left(w_\ceil{t}^* + hLz_\ceil{t}^*\right) \\
			\le& \frac{1}{2}hz_\ceil{t}^*(Lh) +  hz_\ceil{t}^*\left(M\left(v_\ceil{t}^* + h\left(Lx_\ceil{t}^* + L'\right)\right) + L\left(2t(L+1) + hL\right)\right).
		\end{align*}
		In the second inequality we used that $\abs{\cos(t)-1} \le \frac{h^2}{2}$ for the first term and the same bound which we derived in (\ref{equ:calculation_DPhi_borderterms}) for the second term. Since
		\begin{align*}
			&\norm{\eta(t)}_l \le z_\ceil{t}^* \norm{\cos(\bt{t})CA_\floor{t} + CA_\ceil{t} - 2CA_t}_l
		\end{align*}
		this directly implies a bound for $\eta^*(t)$.
		We observe that by (\ref{equ:definition_beta_proof}) and (\ref{equ:same_vel_wbound})
		\begin{align*}
			\beta(t) \le 8t^2(1+L)^2{z_t^*}^2 + \eta^*(t) + \epsilon^*(t),
		\end{align*}
		where $\eta^*(t) = \max_{s \le t}\abs{\eta(s)}$ and $\epsilon^*(t) = \max_{s \le t}\abs{\epsilon(s)}$.
		This leads to a bound for the integral containing $\beta$, given by 
		\begin{align*}
			&\int_0^t \frac{1}{\sqrt{2\zeta}}\sin(\sqrt{2\zeta}(t - r)) \beta(r) \dif r \\
			\le& 8(1+L)^2{z_t^*}^2\int_0^t (t - r)r^2 \dif r + \int_0^t (t - r) \dif r (\epsilon^*(t) + \eta^*(t)) \\
			=& \frac{2}{3}(1+L)^2 t^4 z_t^* + \frac{1}{2}t^2 (\epsilon^*(t) + \eta^*(t)).
		\end{align*}
		Finally note that $\delta(0) = 0$ and therefore the summand containing $\delta(0)$ in (\ref{equ:solution_ode_a}) vanishes.
		Now we put it all together to obtain
		\begin{align*}
			\norm{z_t}_l^2 = a(t) \le \cos(\sqrt{2\zeta}t)\norm{z_0}_l^2 + \left(\frac{2}{3}(1+L)^2t^4 + D_t\right){z_\ceil{t}^*}^2
		\end{align*}
		with $D_t$ given by
		\begin{align*}
			D_t =&\left(th^2 + t^3h^2(L+1) + \frac{1}{2}t^2h\right)\\
			&\quad \times \left(M\left(v_\ceil{t}^* + h\left(Lx_\ceil{t}^* + L'\right)\right) + L\left(2t(L+1) + hL\right)\right) \\
			&+ \left(\frac{1}{6}th^2 + \frac{1}{6}t^3h^2(L+1) + \frac{1}{4}t^2h \right)\left(hL\right)
		\end{align*}
		where we used that $w^*_t \le 2t(1+L)z_t^*$ by \lemref{lemma:zw_prioribounds} and the bounds for $\abs{\delta(t)}, \abs{\eta(t)}$ and $\abs{\epsilon(t)}$. 
		
		$v_t^*, x_t^*$ are both bounded from above by $K(1 + \norm{x}_l + \norm{v}_l)$ by \lemref{lemma:qv_prioribounds} for a constant $K$ depending only on $L$ and $L'$. We use that since \assref{ass:Ltht1} holds, we can upper bound $t$ by a constant depending only on $L$. We can do the same with $h$ since $h \le t$. Therefore there is a constant $K$, depending only on $L$, $L'$ and $M$ such that
		\begin{equation*}
			D_t \le K t^2h (1 + \norm{x}_l + \norm{v}_l).
		\end{equation*}
		We note that $\cos(\sqrt{2\zeta}t) \le 1 - \zeta  t^2 + \frac{1}{6}\zeta^2 t^4$. Furthermore by \assref{ass:LthtzetaL} we have that $\zeta^2t^4 \le (1+L)^2t^4 \le \zeta t^2$. Therefore
		\begin{equation}
			\label{equ:inequality_thatleadsto_contractionbound}
			\begin{aligned}
				\norm{z_t}_l^2 &\le (1 - \zeta t^2)\norm{z_0}_l^2 + \left( \frac{5}{6} \zeta t^2 + K t^2h (1 + \norm{x}_l + \norm{v}_l))\right){z_\ceil{t}^*}^2 \\
				&\le (1 - \zeta t^2)\norm{z_0}_l^2 +\frac{11}{12} \zeta t^2{z_\ceil{t}^*}^2
			\end{aligned}
		\end{equation}
		where we used (\ref{equ:ass_size_bound_x_v}). Note however that the $K$ in the assumptions of this lemma and the $K$ in the equation above differ by a factor of $\frac{1}{12}$.
		We claim that this already implies that
		\begin{equation}
			\label{equ:claim_that_we_have_contraction}
			\norm{z_t}_l^2 \le (1 - \frac{1}{12}\zeta t^2)\norm{z_0}_l^2.
		\end{equation}
		First assume that there is an $s$ such that $\norm{z_s}_l > \norm{z_0}_l$. Let $r \in \intcc{\floor{s}, \ceil{s}}$ such that $z_\ceil{s}^* = \norm{z_r}_l \ge \norm{z_s}_l > \norm{z_0}_l$. Such an $r$ exists since $z$ is continuous and $\intcc{\floor{s}, \ceil{s}}$ is compact. Then by (\ref{equ:inequality_thatleadsto_contractionbound}) we have that
		\begin{align*}
			\norm{z_r}_l^2 &\le (1 - \zeta t^2)\norm{z_0}_l^2 +\frac{11}{12} \zeta t^2{z_\ceil{r}^*}^2 = (1 - \zeta t^2)\norm{z_0}_l^2 +\frac{11}{12} \zeta t^2\norm{z_r}_l^2 \\
			&< (1 - \frac{1}{12}\zeta t^2)\norm{z_r}_l^2 < \norm{z_r}_l^2
		\end{align*}
		which is a contradiction. Therefore $\norm{z_t}_l \le \norm{z_0}_l$ for all $t$ that satisfy \assref{ass:LthtzetaL}. Especially ${z_t^*}^2 \le \norm{z_0}_l^2$ which implies (\ref{equ:claim_that_we_have_contraction}). 
	\end{proof}
	\begin{remark}
		\lemref{lemma:contraction} will be a key estimate in the proof of \thmref{thm:contraction_hs}. The alert reader will have noted that \lemref{lemma:contraction} makes an assumption on the size of  $\norm{v}_l$, while \thmref{thm:contraction_hs} does not. Therefore, for large $v$ we will have to default to the bound (\ref{equ:priori_bound_ws}) from \lemref{lemma:zw_prioribounds}.
	\end{remark}
	Knowing that the dynamics are contractive we will need to bound the probability that one of the proposals is accepted while the other one is not. The next lemma gives us the tools to bound these probabilities. It differs from the proof in \citet{bou2018coupling} where they directly bound the change $\od{H(q_t, v_t)}{t}$. Instead we use the formula for the energy error (\ref{equ:DeltaH_pHMC}) and write it as an approximation of an integral. The energy error will then be the approximation error for which bounds are known.
	\begin{lemma}
		\label{lemma:DH}
		Assume that \assref{ass:DPhiBoundL}, \assref{ass:DPhi2BoundM} and \assref{ass:Ltht1} hold. 
		Then there exists a constant $K_1 > 0$ depending only on $L, L'$ and $M$ such that
		\begin{equation}
			\label{equ:lemma_DH_DHBOUND}
			\abs{\Delta H(q_0, v_0)} \le K_1th^2\left(1 + \max\left\{\norm{q_0}_l, \norm{v_0}_l\right\}^3\right)
		\end{equation}
		and a constant $K_2 > 0$ depending only on $L, L', M$ and $N$ such that 
		\begin{align*}
			\abs{\partial_{(z, w)}\Delta H(q_0, v_0)}\le K_2th^2\max\{\norm{z}_l, \norm{w}_l\} (1  + \max\{\norm{q_0}_l, \norm{v_0}_l\}^3)
		\end{align*}
	\end{lemma}
	\begin{proof}
		Recall the exact form of the energy error (\ref{equ:DeltaH_pHMC}) and the formal limit of $\Delta H$ for $h \to 0$ given as
		\begin{equation*}
			\Delta H(q_0, v_0) \xrightarrow[h \to 0]{} \Phi(q_I) - \Phi(q_0) - \int_0^t \Dif\Phi(q_s)v_s \dif s = 0.
		\end{equation*}
		We will write $\Delta H$ as an approximation of the formal limit. Then the size of the energy error $\Delta H$ will mainly stem from the error we make when approximating the integral $\int_0^T \Dif\Phi(q_s)v_s \dif s$. This error  Let us write $\Delta H$ as
		\begin{align*}
			\Delta H(q_0, v_0) = \frac{h^2}{8}\left(\norm{C^{1/2}\Dif\Phi(q_0)}^2 - \norm{C^{1/2}\Dif\Phi(q_I)}^2\right) + \sum_{i=1}^{I} \epsilon_i(q_0, v_0)
		\end{align*}
		with
		\begin{align*}
			\epsilon_i = \epsilon_i(q_0, v_0) = \Phi(q_i) - \Phi(q_{i-1}) - \frac{h}{2}(\sprod{\Dif \Phi(q_{i-1})}{v_{i-1}} + \sprod{\Dif \Phi(q_{i})}{v_{i}}).
		\end{align*}
		Now we bound $\abs{\epsilon_i}$. We again use the interpolation (\ref{equ:ode_discr_interpolation_nonlinear}) and view the discrete iterates $(q_i, v_i)$ as the evaluation of continuous paths $(q_s)_{s \in [0, t]}$ and $(v_s)_{s \in [0, t]}$ at times $ih$, i.e. $(q_i, v_i) = (q_{ih}, v_{ih})$ for $i=0,\ldots, I = \frac{t}{h}$ and see
		\begin{align*}
			\Phi(q_i) - \Phi(q_{i-1}) =& \Phi(q_{ih}) - \Phi(q_{(i-1)h}) \\
			=& \int_{(i-1)h}^{ih} \dod{\Phi(q_s)}{s} \dif s \\
			=& \int_{(i-1)h}^{ih} \Dif\Phi(q_s)\dot{q_s} \dif s \\
			=& \int_{(i-1)h}^{ih} \Dif\Phi(q_s)\left(v_s - \frac{\sin(\bt{s})}{2} C\Dif\Phi(q_{\floor{s}})
			+ \frac{\bt{s}}{2}C \Dif \Phi(q_{\ceil{s}})\right) \dif s.
		\end{align*}
		Therefore
		\begin{align*}
			\abs{\epsilon_i} \le \delta_i + \gamma_i
		\end{align*}
		where
		\begin{align*}
			\delta_i =& \abs{\int_{(i-1)h}^{ih} \Dif\Phi(q_s)^tv_s \dif s - \frac{h}{2}\left(\Dif \Phi(q_{(i-1)h})^t{v_{(i-1)h}} + {\Dif \Phi(q_{ih})}^t{v_{ih}}\right)}\\
			\gamma_i =&\abs{\int_{(i-1)h}^{ih} \Dif\Phi(q_s)\left(\frac{\bt{s}}{2}C \Dif \Phi(q_{\ceil{s}}) -\frac{\sin(\bt{s})}{2} C\Dif\Phi(q_{\floor{s}})\right) \dif s}.
		\end{align*}
		By $q^t$ we denote the transpose of a finite dimensional vector.
		For $\delta_i$ we can use the approximation result that for $f \in C^2([a, b])$
		\begin{equation*}
			\int_c^d f(x) \dif x - \frac{d-c}{2}(f(c) + f(d)) = -\frac{h^3}{12}f''(s)
		\end{equation*}
		with $s \in \intcc{c, d}$ (see Theorem 9.4 of \citet{kress1998numerical}). Since $q_s, v_s$ are smooth on $[ih, (i+1)h]$ and $\Phi \in C^4$ we can use above result and see
		to see that
		\begin{align*}
			\delta_i = \abs{\frac{h^3}{12} \dod[2]{}{t}{\Dif\Phi(q_s)^t}{v_s}}
		\end{align*}
		for an $s \in \intoo{(i-1)h, ih}$. We denote by $\dot{f(t)} = \od{f(t)}{t}$ the time derivative. $\ddot{f(t)}$ denotes the second time derivative. We calculate
		\begin{align*}
			\dod[2]{}{t}{\Dif\Phi(q_s)^t}{v_s} = {\ddot{\Dif\Phi(q_s)}}^t{v_s} + 2{\dot{\Dif\Phi(q_s)}}^t{\dot{v_s}} + {\Dif\Phi(q_s)}^t{\ddot{v_s}}.
		\end{align*}
		Now we calculate and bound each of the terms. We start by calculating the second time derivatives of $q_s$ and $v_s$,
		\begin{equation}
			\label{equ:second_order_derv}
			\begin{aligned}
				\ddot{q_s} &= -q_s - \cos(\bt{s})C\Dif\Phi(q_\floor{s}), \text{~ and}\\
				\ddot{v_s} &= -v_s + \sin(\bt{s})C\Dif\Phi(q_\floor{s}) - \frac{\bt{s}}{2} C\Dif\Phi(q_\ceil{s}).
			\end{aligned}
		\end{equation}
		We bound them by
		\begin{align*}
			\norm{\ddot{q_s}}_l &\le q_\ceil{s}^*(1 + L) + L'\\
			\norm{\ddot{v_s}}_l &\le v_\ceil{s}^* + 2h(Lq_\ceil{s}^* + L')
		\end{align*}
		with constants $K$ only depending on $L$ and $L'$. Now
		\begin{align*}
			\norm{\ddot{\Dif\Phi(q_s)}}_l
			=& \norm{\dod{\Dif^2\Phi(q_s)\dot{q_s}}{t}}_l\\
			=& \norm{\Dif^3\Phi(q_s)(\dot{q_s}, \dot{q_s}) + \Dif^2\Phi(q_s)\ddot{q_s}}_l\\
			\le& M\left(v_\ceil{s}^* + h(Lq_\ceil{s}^* + L')\right)^2 + L\left(q_\ceil{s}^*(1 + L) + L'\right)
		\end{align*}
		and therefore
		\begin{align*}
			\abs{\dod[2]{}{t}{\Dif\Phi(q_s)^t}{v_s}} \le& v_\ceil{s}^*\left(M\left(v_\ceil{s}^* + h(Lq_\ceil{s}^* + L')\right)^2 + L\left(q_\ceil{s}^*(1 + L) + L'\right)\right) \\
			&+ 2L\left(v_\ceil{s}^* + h(Lq_\ceil{s}^* + L')\right)\left(q_\ceil{s}^*(L+1) + L'\right) \\
			&+ \left(Lq_\ceil{s}^*	+ L'\right)\left(v_\ceil{s}^* + 2h(Lq_\ceil{s}^* + L')\right).
		\end{align*}
		Especially
		\begin{align*}
			\delta_i \le \frac{h^3}{12} \abs{\dod[2]{}{t}{\Dif\Phi(q_s)^t}{v_s}} \le h^3K(1 + \max\{\norm{q_0}_l, \norm{v_0}_l\}^3) =: \delta^*
		\end{align*}
		for a constant $K > 0$ depending only on $L, L'$ and $M$. The bound on the right hand side exists because $s \le t \in h\bb{Z}$ and we can bound $q_t^*$ and $v_t^*$ by a multiple of $\norm{q_0}_l$ and $\norm{v_0}_l$ that only depends on $L$ and $L'$ by \lemref{lemma:qv_prioribounds}. For $\gamma_i$ we get
		\begin{align*}
			\gamma_i \le& (Lq_t^* + L')\int_{(i-1)h}^{ih}\norm{\frac{\bt{s}}{2}C \Dif \Phi(q_{ih}) -\frac{\sin(\bt{s})}{2} C\Dif\Phi(q_{(i-1)h})}_l \dif s.
		\end{align*}
		We now calculate that
		\begin{align*}
			&\norm{\frac{\bt{s}}{2}C \Dif \Phi(q_{ih}) -\frac{\sin(\bt{s})}{2} C\Dif\Phi(q_{(i-1)h})}_l \\
			\le& \norm{\frac{\bt{s}}{2}\left(C\Dif \Phi(q_{ih}) -C\Dif\Phi(q_{(i-1)h})\right)}_l + \abs{\frac{\bt{s}}{2} - \frac{\sin(\bt{s})}{2}}\norm{C\Dif\Phi(q_{(i-1)h})}_l \\
			\le& \frac{\bt{s}}{2} L \norm{q_{ih} - q_{(i-1)h}}_l + \frac{1}{12}\bt{s}^3(Lq_t^* + L') \\
			\le& \frac{\bt{s}}{2}Lh(v_\ceil{s}^* + h(Lq_\ceil{s}^* + L')) + \frac{1}{12}\bt{s}^3(Lq_\ceil{s}^* + L')
		\end{align*}
		where we used (\ref{equ:x_sin_difference_bound}) to bound $\bt{s} - \sin(\bt{s})$. Therefore
		\begin{align*}
			\gamma_i \le& (Lq_t^* + L')\left(\frac{h^3}{4}L\left(v_t^* + h\left(Lq_t^* + L'\right)\right) + \frac{h^4}{48}\left(Lq_t^* + L'\right)\right) \\
			\le& h^3K(1 + \max\{\norm{q_0}_l + \norm{v_0}_l\}^2) =: \gamma^*
		\end{align*}
		with a constant $K > 0$ depending only on $L$ and $L'$. The bound again exists because of \lemref{lemma:qv_prioribounds}. 
		We conclude that
		\begin{align*}
			\sum_{i=1}^I \epsilon_i \le \frac{t}{h}(\delta^* + \gamma^*) \le th^2K\left(1 + \max\left\{\norm{q_0}_l, \norm{v_0}_l\right\}^3\right)
		\end{align*}
		for some constant $K > 0$ that only depends on $L$ and $L'$. It only depends on $L$ and $L'$ since by \lemref{lemma:qv_prioribounds} both $q_t^*$ and $v_t^*$ are bounded by a constant multiple of $\max\{\norm{q_0}_l, \norm{v_0}_l\}$ because $t$ is bounded by \assref{ass:Ltht1}. We also calculate that
		\begin{align*}
			&\abs{\norm{C^{1/2}\Dif\Phi(q_I)}^2 - 
				\norm{C^{1/2}\Dif\Phi(q_0)}^2}\\
			=& \abs{\int_0^t 2\sprod{C^{1/2}\Dif^2\Phi(q_s)\dot{q_s}}{C^{1/2}\Dif\Phi(q_s)} \dif s} \\
			\le& 2t L\left(v_t^* + h(Lq_t^* + L')\right)\left(L(q_t^* + L')\right)
		\end{align*}
		and therefore
		\begin{align*}
			\frac{h^2}{8}\left(\norm{C^{1/2}\Dif\Phi(q_0)}^2 - \norm{C^{1/2}\Dif\Phi(q_I)}^2\right) \le h^2t K\left(1 + \max\left\{\norm{v_0}, \norm{q_0}\right\}^2\right).
		\end{align*}
		Putting it all together we obtain
		\begin{align*}
			\abs{\Delta H(q_0, v_0)} \le th^2K\left(1 + \max\left\{\norm{q_0}_l, \norm{v_0}_l\right\}^3\right)
		\end{align*}
		for a constant $K > 0$ depending only on $L, L'$ and $M$. This proves the first claim.
		
		We now do the same for the derivative of $\Delta H$. We denote $\partial_{(z, w)} q_{s}(q_0, v_0)$ as $q_s'$ and similarly for $v_s$.  $\partial_{(z, w)} \Delta H(q_0, v_0)$ is denoted by $\Delta H'$. First we need to show a priori bounds for $q_s'$ and $v_s'$ similar to \lemref{lemma:qv_prioribounds}. Note that $q_s'$ and $v_s'$ fulfil the differential equations
		\begin{equation*}
			\begin{cases}
				\dod{q_s'}{s} &= 
				v_s'
				- \frac{\sin(\bt{s})}{2} C\Dif^2\Phi(q_{\floor{s}})q_{\floor{s}}'
				+ \frac{\bt{s}}{2}C \Dif^2 \Phi(q_{\ceil{s}})q_{\ceil{s}}' \\
				\dod{v_s'}{s} &= -q_s' - \frac{\cos(\bt{s})}{2}C\Dif^2\Phi(q_\floor{s})q_\floor{s}' - \frac{1}{2}C\Dif^2\Phi(q_\ceil{s})q_\ceil{s}'
			\end{cases}
		\end{equation*}
		with initial conditions $(q_0', v_0') = (z, w)$. As in the proofs of \lemref{lemma:qv_prioribounds} and \lemref{lemma:zw_prioribounds} we calculate
		\begin{align*}
			&{\int_0^t v_r' \dif r}\\
			=& {tv_0' + \int_0^t \int_0^r -q_u' - \frac{\cos(\bt{u})}{2}C\Dif^2\Phi(q_\floor{u})q_\floor{u}' - \frac{1}{2}C\Dif^2\Phi(q_\ceil{u})q_\ceil{u}' \dif u \dif r}
		\end{align*}
		and bound
		\begin{align*}
			&\norm{\int_0^t \int_0^r -q_u' - \frac{\cos(\bt{u})}{2}C\Dif^2\Phi(q_\floor{u})q_\floor{u}' - \frac{1}{2}C\Dif^2\Phi(q_\ceil{u})q_\ceil{u}' \dif u \dif r}_l \\
			\le& \frac{1}{2}t^2q_t'^*(1 + L)
		\end{align*}
		where we used that $t \in \bb{Z}$. We now see that
		\begin{align*}
			\norm{q_t' - q_0' - tv_0'}_l \le \frac{1}{2}t^2 q_{t}'^*(1 + L) + Lthq_{t}'^* = \frac{1}{2}q_{t}'^*(t^2 + L(t^2 + 2th)) 
		\end{align*}
		from which we conclude using \assref{ass:Ltht1} that
		\begin{align*}
			\norm{q_t' - q_0' - tv_0'}_l &\le (t^2 + L(t^2 + 2th)) \max\{\norm{q_0'}_l, \norm{q_0' + tv_0'}_l\}
		\end{align*}
		and
		\begin{align*}
			{q_t'}^* &\le (1 + t^2 + L(t^2 + 2th)) \max\{\norm{q_0'}_l, \norm{q_0' + tv_0'}_l\} \\
			&\le 2\max\{\norm{q_0'}_l, \norm{q_0' + tv_0'}_l\}.
		\end{align*}
		We again see that 
		\begin{align*}
			\norm{v_t' - v_0'}_l &\le t {q_t'}^*(1 + L)
		\end{align*}
		and therefore obtain the same bounds as in \lemref{lemma:zw_prioribounds}.
		
		Now we turn to the exact formula for $\Delta H'$ which is
		\begin{align*}
			&\frac{h^2}{8}\left(2 \left(C^{1/2}\Dif\Phi(q_0)\right)^tC^{1/2}\Dif^2\Phi(q_0)q_0' - 2 \left(C^{1/2}\Dif\Phi(q_t)\right)^tC^{1/2}\Dif^2\Phi(q_t)q_t'\right) \\
			+&\sum_{i=1}^I \epsilon_i,
		\end{align*}
		where
		\begin{align*}
			\epsilon_i =& \Dif \Phi(q_{ih})^tq_{ih}' - \Dif \Phi(q_{(i-1)h})^t q_{(i-1)h}' \\
			&- \frac{h}{2}\left(
			v_{(i-1)h}^t \Dif^2 \Phi(q_{(i-1)h}) q_{(i-1)h}'
			+\Dif\Phi(q_{(i-1)h})^t v_{(i-1)h}'\right) \\
			&-\frac{h}{2}\left(v_{ih}^t \Dif^2 \Phi(q_{ih}) q_{ih}' 
			+\Dif\Phi(q_{ih})^t v_{ih}'
			\right).
		\end{align*}
		We now see that 
		\begin{align*}
			&\Dif \Phi(q_{ih})^tq_{ih}' - \Dif \Phi(q_{(i-1)h})^t q_{(i-1)h}' 
			= \int_{(i-1)h}^{ih} \left(\Dif^2\Phi(q_s)\dot{q_s}\right)^t q_{s}' + \Dif\Phi(q_s)^t \dot{q_s}' \dif s
		\end{align*}
		and therefore
		\begin{align*}
			\epsilon_i = \delta_i + \gamma_i
		\end{align*}
		where
		\begin{align*}
			\delta_i =& \int_{(i-1)h}^{ih} v_s^t \Dif^2 \Phi(q_s)q_s' + \Dif\Phi(q_s)^t v_s' \dif s \\
			&- \frac{h}{2}\left(
			v_{(i-1)h}^t \Dif^2 \Phi(q_{(i-1)h}) q_{(i-1)h}'
			+\Dif\Phi(q_{(i-1)h})^t v_{(i-1)h}'\right) \\
			&-\frac{h}{2}\left(v_{ih}^t \Dif^2 \Phi(q_{ih}) q_{ih}' 
			+\Dif\Phi(q_{ih})^t v_{ih}'
			\right)
		\end{align*}
		and
		\begin{align*}
			\gamma_i =& \int_{(i-1)h}^{ih} \left(-\frac{\sin(\bt{s})}{2} C\Dif\Phi(q_{\floor{s}})
			+ \frac{\bt{s}}{2}C \Dif \Phi(q_{\ceil{s}})\right)^t \Dif^2\Phi(q_s)q_s' \dif s \\
			&+ \int_{(i-1)h}^{ih} \Dif\Phi(q_s)^t \left(-\frac{\sin(\bt{s})}{2} C\Dif^2\Phi(q_{\floor{s}})q_{\floor{s}}'
			+ \frac{\bt{s}}{2}C \Dif^2 \Phi(q_{\ceil{s}})q_{\ceil{s}}'\right) \dif s.
		\end{align*}
		$\delta_i$ is again the hardest to bound. We again use Theorem 9.4  of \citet{kress1998numerical} and need to bound
		\begin{align*}
			&\abs{\dod[2]{}{t} v_s^t \Dif^2 \Phi(q_s)q_s' + \Dif\Phi(q_s)^t v_s'}.
		\end{align*}
		We start by bounding
		\begin{equation}
			\label{equ:secder_first_termbla}
			\begin{aligned}
				&{\dod[2]{}{t} v_s^t \Dif^2 \Phi(q_s)q_s'}\\
				=~& \ddot{v_s}^t\Dif^2\Phi(q_s)q_s' + v_s^t\left(\Dif^4\Phi(q_s)(\dot{q_s}, \dot{q_s}) + \Dif^3\Phi(q_s)\ddot{q_s}\right)q_s' + v_s^t \Dif^2\Phi(q_s)\ddot{q_s}'\\
				&+ 2\dot{v_s}^t\Dif^3\Phi(q_s)(\dot{q_s})q_s' + 2\dot{v_s}^t\Dif^2\Phi(q_s)\dot{q_s}' + 2 v_s^t\Dif^3(q_s)(\dot{q_s})\dot{q_s}'.
			\end{aligned}
		\end{equation}
		Analogously to (\ref{equ:second_order_derv}) we see that
		\begin{align*}
			\ddot{q_s}' = -q_s' - \cos(\bt{s})C\Dif^2\Phi(q_\floor{s})q_\floor{s}'
		\end{align*}
		and
		\begin{align*}
			\norm{\ddot{q_s}'}_l \le {q_s'}^*(1 + L).
		\end{align*}
		Now we bound the terms in (\ref{equ:secder_first_termbla}):
		\begin{align*}
			&\abs{\dod[2]{}{t} v_s^t \Dif^2 \Phi(q_s)q_s'} \\
			\le& \left(v_\ceil{s}^* + 2h(Lq_\ceil{s}^* + L')\right)Lq_\ceil{s}'^* \\
			&+ v_\ceil{s}^*\left(N\left(v_\ceil{s}^* + h(Lq_\ceil{s}^* + L')\right)^2 + M\left(q_\ceil{s}^*(1 + L) + L'\right)\right)q_\ceil{s}'^* \\
			&+ v_\ceil{s}^*L\left(q_\ceil{s}'^*(1 + L)\right) \\
			&+ 2\left(q_\ceil{s}^*(1 + L) + L'\right)M\left(v_\ceil{s}^* + h(Lq_\ceil{s}^* + L')\right)q_\ceil{s}'^* \\
			&+ 2\left(q_\ceil{s}^*(1 + L) + L'\right)L\left(v_\ceil{s}'^* + hLq_\ceil{s}'^*\right) \\
			&+ 2v_\ceil{s}^* M\left(v_\ceil{s}^* + h(Lq_\ceil{s}^* + L')\right)\left(v_\ceil{s}'^* + hLq_\ceil{s}'^*\right).
		\end{align*}
		We see that we can bound
		$\ceil{s}$
		\begin{equation}
			\label{equ:bound_complicated_term_k}
			\abs{\dod[2]{}{t} v_t^t \Dif^2 \Phi(q_t)q_t'} \le K\max\{\norm{q_0'}_l, \norm{v_0'}_l\}\left(1 + \max\{\norm{q_0}_l, \norm{v_0}_l\}^3\right)
		\end{equation}
		for a constant $K > 0$ depending only on $L, L', M, N$. 
		
		Now we bound
		\begin{align*}
			{\dod[2]{}{t} \Dif\Phi(q_s)^t v_s'} =& \Dif^3\Phi(q_s)(\dot{q_s}, \dot{q_s})^t v_s' + \Dif^2\Phi(q_s)(\ddot{q_s})^t v_s' \\
			&+ 2\Dif^2\Phi(q_s)(\dot{q_s})\dot{v_s}' + \Dif\Phi(q_s)\ddot{v_s}'
		\end{align*}
		by
		\begin{align*}
			\abs{\dod[2]{}{t} \Dif\Phi(q_s)^t v_s'} \le& M\left(v_\ceil{s}^* + h(Lq_\ceil{s}^* + L')\right)^2v_\ceil{s}'^* \\
			&+ L\left(q_\ceil{s}^*(1 + L) + L'\right)v_\ceil{s}'^*\\
			&+ 2L\left(v_\ceil{s}^* + h(Lq_\ceil{s}^* + L')\right)\left(q_\ceil{s}'^*(1 + L)\right) \\
			&+ (Lq_\ceil{s}^* + L')\left(v_\ceil{s}'^* + 2hLq_\ceil{s}'^*\right)\\
			\le& K\max\{\norm{q_0'}_l, \norm{v_0'}_l\}\left(1 + \max\{\norm{q_0}_l, \norm{v_0}_l\}^2\right)
		\end{align*}
		for a constant $K > 0$ depending only on $L, L', M$. These bounds especially lead to
		\begin{equation*}
			\abs{\delta_i} \le h^3K\max\{\norm{q_0'}_l, \norm{v_0'}_l\}\left(1 + \max\{\norm{q_0}_l, \norm{v_0}_l\}^3\right).
		\end{equation*}
		We can also bound $\gamma_i$:
		\begin{align*}
			\abs{\gamma_i} \le& {q_t'}^*L \left(
			\frac{h^4}{48} \left(Lq_t^* + L'\right) + \frac{h^3}{4}L\left(v_t^* + h(Lq_t^* + L')\right)\right) \\
			&+ \left(Lq_t^* + L'\right)\left(
			\frac{h^4}{48}L{q_t'}^* + \frac{h^3}{4}M\left({v_t}^* + h(Lq_t^*+L')\right){q_t'}^*\right)\\
			&+ \left(Lq_t^* + L'\right)\left(\frac{h^3}{4}L\left({v_t'}^* + hL{q_t'}^*\right)
			\right) \\
			\le& h^3K\max\{\norm{q_0'}_l, \norm{v_0'}_l\}\left(1 + \max\{\norm{q_0}_l, \norm{v_0}_l\}^2\right)
		\end{align*}
		where $K$ depends only on $L, L'$ and $M$. Therefore
		\begin{equation}
			\label{equ:inequ_eps_i_laisdh}
			\begin{aligned}
				&\sum_{i=1}^I \abs{\epsilon_i} 
				\le th^2K\max\{\norm{q_0'}_l, \norm{v_0'}_l\} \left(1  + \max\{\norm{q_0}_l, \norm{v_0}_l\}^3\right)
			\end{aligned}
		\end{equation}
		for some constant $K$ depending only on $L, L', M$, and $N$. At last we now bound
		\begin{align*}
			&\left(C^{1/2}\Dif\Phi(q_0)\right)^tC^{1/2}\Dif^2\Phi(q_0)q_0' -  \left(C^{1/2}\Dif\Phi(q_t)\right)^tC^{1/2}\Dif^2\Phi(q_t)q_t'\\
			\le&  {q_t'}^*L^2t\left(v_t^* + h(Lq_t^* + L')\right) 
			+ \left(Lq_t^* + L'\right){q_t'}^* Mt\left(v_t^* + h(Lq_t^* + L')\right)\\
			&+ \left(Lq_t^* + L'\right)Lt\left({v_t'}^* + hL{q_t'}^*\right)
		\end{align*}
		by splitting the terms up according to 
		\begin{equation*}
			a_1b_1c_1 - a_2b_2c_2 = (a_1 - a_2)b_1c_1 + (b_1 - b_2)a_2c_1 + (c_1 - c_2)a_2b_2.
		\end{equation*}
		We see that now an analogous bound to (\ref{equ:inequ_eps_i_laisdh}) holds and therefore it also holds for $\Delta H'$ which concludes the proof. 
	\end{proof}
	We are now ready to prove the finite dimensional version of \thmref{thm:exact_pHMC_contractive_hs} and \thmref{thm:contraction_hs}.
	\begin{theorem}
		\label{thm:exact_pHMC_contractive_Rd}
		Assume that \assref{ass:DPhiBoundL}, \assref{ass:PhiConvex} and \assref{ass:DPhi2BoundM} hold. Assume that $T$ is such that for $h=0$, \assref{ass:LthtzetaL} holds, i.e. $T^2(L + 1) \le \frac{\zeta}{L + 1}$. Then for every $x, y \in \mathcal{H}_N \approx \bb{R}^N$ 
		\begin{equation*}
			\bb{E}\left[\norm{x'(x, v) - y'(y, v)}_l\right] \le \left(1 - \frac{1}{27}\zeta T^2\right) \norm{x - y}_l.
		\end{equation*}
	\end{theorem}
	\begin{proof}
		This directly follows from \lemref{lemma:contraction} since $x'(x, v) = \varphi_T(x, v) = \psi^{(T)}_0(x, v)$.
	\end{proof}
	To make the notation a bit cleaner, we define
	\begin{equation*}
		R'(x, y) = \norm{x'(x, v) - y'(y, v)} \text{~ and ~} r(x, y) = \norm{x - y}.
	\end{equation*}
	\begin{theorem}
		\label{thm:contraction_Rd}
		Assume that \assref{ass:PhiConvex}, \assref{ass:DPhiBoundL} and \assref{ass:DPhi2BoundM} hold. Assume that $T$ and $h_1$ are such that \assref{ass:LthtzetaL} is satisfied for $h = h_1$. Furthermore, let $R$ be a real number such that $\bb{P}[\norm{v}_l > R] \le \frac{\zeta}{2000(L+1)}$. Then there exists an $h_0$ such that for any $0 < h \le \min\{h_0, h_1\}$ with $T/h \in \bb{Z}$ and for any $x, y \in \bb{R}^d$ with $\max\{\norm{x}_l, \norm{y}_l\} \le R$ we have
		\begin{equation*}
			\bb{E}[R'(x, y)] \le \left(1 - \frac{1}{27}\zeta T^2\right) r(x,y).
		\end{equation*}
		Furthermore for fixed $L, L', M$ and $C$,  $h_0^{-1}$ has a lower bound of the form $\mathcal{O}((1 + T^{-1/2})(1 + R^2))$.
	\end{theorem}
	\begin{proof}
		Remember that we use the same uniform random variable $U \sim \text{Unif}(0,1)$ to determine if the move from $x$ or $y$ is accepted or rejected, see the discussion before \thmref{thm:contraction_hs}. We define the acceptance event when the dynamics are started in $x$ as
		\begin{equation*}
			A(x) =\{U  \le \exp(-\Delta H(x, v)\}.
		\end{equation*}
		We see that
		\begin{align*}
			R'(x, y) &= \norm{q_T(x, v) - q_T(y, v)}_l && \text{on } A(x) \cap A(y)\\
			R'(x, y) &= r(x, y) && \text{on } A(x)^C \cap A(y)^C.
		\end{align*}
		On $A(x) \cap A(y)^C$ we have $y' = y$ and thus
		\begin{equation*}
			R'(x, y) - r(x, y) = \norm{q_T(x, v) - y}_l - \norm{x - y}_l \le \norm{q_T(x, v) - x}_l
		\end{equation*}
		and by a similar argumentation on $A(x)^C \cap A(y)$
		\begin{equation*}
			R'(x, y) - r(x, y) \le \norm{q_T(y, v) - y}_l.
		\end{equation*}
		Therefore
		\begin{equation}
			\label{equ:aufsplitten_terms_Rr}
			\begin{aligned}
				\bb{E}\left[R'(x, y) - r(x, y)\right] \le& \bb{E}\left[\norm{q_T(x, v) - q_T(y, v)}_l - \norm{x - y}_l; A(x) \cap A(y)\right] \\
				&+ \bb{E}\left[\norm{q_T(x, v) - x}_l; A(x) \cap A(y)^C\right] \\
				&+ \bb{E}\left[\norm{q_T(y, v) - y}_l; A(x)^C \cap A(y)\right].
			\end{aligned}
		\end{equation}
		We start by bounding the first term. Choose $K \in \intoo{0, \infty}$ as in \lemref{lemma:contraction} and assume that $h \le \min\{h_1, h_2\}$ where $h_2 = \frac{\zeta}{K(1 + 2R)}$. Then by \lemref{lemma:contraction} we have
		\begin{equation*}
			\norm{q_T(x, v) - q_T(y, v)}_l \le (1 - \frac{1}{24}\zeta T^2)\norm{x - y}_l
		\end{equation*}
		where we used that $\sqrt{(1 - \frac{1}{12}\zeta T^2)} \le (1 - \frac{1}{24}\zeta T^2)$. Therefore by \lemref{lemma:zw_prioribounds}
		\begin{align*}
			&\bb{E}\left[\norm{q_T(x, v) - q_T(y, v)}_l - \norm{x - y}_l; A(x) \cap A(y)\right] \\
			\le& -\frac{1}{24}\zeta T^2r(x, y)\bb{P}\left[A(x) \cap A(y) \cap \{\norm{v}_l \le R\}\right]\\
			&+ (T^2(L+1) + 2LTh)r(x, y)\bb{P}\left[\norm{v}_l > R\right].
		\end{align*}
		We now use that
		\begin{align*}
			&-\bb{P}\left[A(x) \cap A(y) \cap \{\norm{v}_l \le R\}\right] \le -\bb{P}\left[A(x) \cap A(y)\right] + 1 - \bb{P}\left[\{\norm{v}_l \le R\}\right]
		\end{align*}
		to see
		\begin{align*}
			&\bb{E}\left[\norm{q_T(x, v) - q_T(y, v)}_l - \norm{x - y}_l; A(x) \cap A(y)\right]\\
			\le& -\frac{1}{24}\zeta T^2r(x, y)\bb{P}\left[A(x) \cap A(y)\right] + (\frac{25}{24}T^2(L+1) + 2LTh)r(x, y)\bb{P}\left[\norm{v}_l > R\right]\\
			\le& -\frac{1}{24}\zeta T^2r(x, y)\bb{P}\left[A(x) \cap A(y)\right] + (\frac{73}{24}T^2(L+1))r(x, y)\bb{P}\left[\norm{v}_l > R\right].
		\end{align*}
		We now first want to lower bound the probability $\bb{P}[A(x) \cap A(y)]$. We do this by upper bounding $\bb{P}[A(x)^c]$.
		For $h \le h_1$ \assref{ass:Ltht1} holds since it is implied by \assref{ass:LthtzetaL}. Therefore we can use (\ref{equ:lemma_DH_DHBOUND}) and see that
		\begin{align*}
			\bb{P}\left[A(q)^C | v\right] &= \abs{1 - \exp(-\Delta H(q, v)^{+})} \\
			&\le \Delta H(q, v)^{+} \\
			&\le K_1Th^2\left(1 + \max\left\{\norm{q_0}_l, \norm{v_0}_l\right\}^3\right).
		\end{align*}
		We now take the expectation to conclude
		\begin{align*}
			\bb{P}\left[A(q)^C\right] &\le KTh^2\left(1 + \bb{E}\left[\max\left\{\norm{q}_l, \norm{v}_l\right\}^3\right]\right) \\
			&\le  KTh^2\left(1 + \norm{q}_l^3 + K\right) \\
			&= KTh^2(1 + \norm{q}_l^3)
		\end{align*}
		where $K$ is some constant that changes from occurrence to occurrence and only depends on $L, L'$ and $M$. The expectation $\bb{E}[\norm{v}_l^3]$, taken with respect to $\mathcal{N}(0, C_N)$, is bounded independent of the embedding dimension $N$, since we can upper bound it by the infinite dimensional integral, taken with respect to $\mathcal{N}(0, C)$, which is finite due to Fernique's theorem (see \citet{hairer_spde}). 
		
		Therefore if $\norm{x}_l, \norm{y}_l \le R$
		\begin{equation*}
			\bb{P}\left[A(x)^C\right] + \bb{P}\left[A(y)^C\right] \le 2KTh^2(1 + R^3),
		\end{equation*}
		where $K$ only depends on the dimensionless quantities $L, L', M$ and $C$. We choose $h_3 > 0$ such that for $h \le h_3$, the expression on the right hand side is smaller  than $\frac{1}{25}$. With such an $h$ we obtain
		\begin{align*}
			\bb{P}\left[A(x) \cap A(y)\right] \ge 1 - (\bb{P}\left[A(x)^C\right] + \bb{P}\left[A(y)^C\right]) \ge 1 - \frac{1}{24} = \frac{24}{25}.
		\end{align*} 
		Therefore for $h \le \min\{h_1, h_2, h_3\}$
		\begin{align*}
			&\bb{E}\left[\norm{q_T(x, v) - q_T(y, v)}_l - \norm{x - y}_l; A(x) \cap A(y)\right]\\
			\le& -\frac{1}{25}\zeta T^2r(x, y) + (\frac{73}{24}T^2(L+1))r(x, y)\bb{P}\left[\norm{v}_l > R\right]\\
			\le& -\frac{1}{26}\zeta T^2r(x, y) 
		\end{align*}
		where we used that
		\begin{equation}
			\label{equ:bound_R2_prob}
			\bb{P}\left[\norm{v}_l > R\right] \le \frac{\zeta}{2000(L+1)}.
		\end{equation}
		(\ref{equ:bound_R2_prob}) in the last inequality. To bound the second term in (\ref{equ:aufsplitten_terms_Rr}) we bound the probability that we land in $A(x)\Delta A(y) = \left(A(x) \cap A(y)^C\right) \cup \left(A(x)^C \cap A(y)\right)$. For that task we use \lemref{lemma:DH} and calculate
		\begin{equation}
			\label{equ:bound_sym_diff}
			\begin{aligned}
				&\bb{P}[A(x) \Delta A(y) | v] \\
				=& \abs{\exp(-\Delta H(x, v)^+) - \exp(-\Delta H(y, v)^+)}\\
				\le& \abs{\Delta H(x, v) - \Delta H(y, v)}\\
				\le& \int_0^1 \abs{\partial_{(y - x, 0)}\Delta H(x_u, v)} \dif u \\
				\le& K_2Th^2\norm{y - x}_l \left(1  + \int_0^1 \max\{\norm{x_u}_l, \norm{v}_l\}^3 \dif u\right) \\
				\le& K_2Th^2\norm{y - x}_l \left(1  + \max\{\norm{x}_l, \norm{y}_l, \norm{v}_l\}^3 \right)
			\end{aligned}
		\end{equation}
		where we have set $x_u = (1-u)x + uy$. 
		We note that by \lemref{lemma:qv_prioribounds}
		\begin{align*}
			\norm{q_T(x, v) - x}_l \le& T\norm{v}_l + \max\{\norm{x}_l, \norm{x + Tv}_l\} + L'(T^2 + 2Th) \\
			\le& \norm{x}_l + 2T\norm{v}_l + L'(T^2 + 2Th)
		\end{align*}
		since $T$ satisfies \assref{ass:Ltht1}. After making a similar computation for $\norm{q_T(y, v) - y}_l$ we can bound the second and third term in (\ref{equ:aufsplitten_terms_Rr}) by
		\begin{equation}
			\label{equ:final_bound_term_2_3}
			\begin{aligned}
				&\bb{E}\left[\norm{q_T(x, v) - x}_l; A(x) \cap A(y)^C\right] + \bb{E}\left[\norm{q_T(y, v) - y}_l; A(x)^C \cap A(y)\right] \\
				\le& \bb{E}\left[\max\{\norm{x}_l, \norm{y}_l\} + L'(T^2 + 2Th) + 2T\norm{v}_l; A(x) \Delta A(y)\right] \\
				\le& \bb{E}[\left((1 + 2T)\max\{\norm{x}_l, \norm{y}_l, \norm{v}_l\} + L'(T^2 + 2Th)\right)\\
				&\qquad\times KTh^2\norm{y - x}_l \left(1  + \max\{\norm{x}_l, \norm{y}_l, \norm{v}_l\}^3 \right)]\\
				&\le \left(1 + 2T + L'(T^2 + 2Th)\right)\left(1 + \bb{E}[\max\{\norm{x}_l, \norm{y}_l, \norm{v}_l\}^4]\right)KTh^2r(x, y) \\
				&\le K\left(1 + 2T + L'(T^2 + 2Th)\right)\left(1 + R^4\right)Th^2r(x, y) \\
			\end{aligned}
		\end{equation}
		where we used (\ref{equ:bound_sym_diff}) in the second inequality and the integrability of $\norm{v}_l^3$ due to Fernique's theorem in the last inequality. $K$ depends on $L, L', M, N$ and $C$. We now choose $h_4$ such that the right hand side of (\ref{equ:final_bound_term_2_3}) is smaller than $\frac{1}{702}\zeta T^2 r(x, y)$. Then 
		\begin{align*}
			\bb{E}[R'(x, y) - r(x, y)] \le -\frac{1}{26}\zeta T^2r(x, y) + \frac{1}{702}\zeta T^2 r(x, y) \le -\frac{1}{27} \zeta T^2 r(x, y)
		\end{align*}
		which concludes the proof for $h_0 = \min\{h_2, h_3, h_4\}$. At last we now consider the dependence of the $h_i$ on $T, h$, and $R$. $h_2^{-1}$ has a dependence of the form $\mathcal{O}({1 + R})$. $h_3^{-1}$ has a dependence of the form $\mathcal{O}(\sqrt{T(1 +R^3)})$. $h_4^{-1}$ has a dependence of the form $\mathcal{O}({\sqrt{\left(1 + T^{-1}\right)(1 + R^4)}})$. Therefore $h_0^{-1}$ has to be of order $\mathcal{O}((1 + T^{-1/2})(1 + R^2))$. 
	\end{proof}
	\subsubsection{Proofs on the Hilbert space}
	Now we prove \thmref{thm:exact_pHMC_contractive_hs} and \thmref{thm:contraction_hs}. We already saw a proof in finite dimensions of \thmref{thm:contraction_hs}. \lemref{lemma:contraction} is basically the proof of \thmref{thm:exact_pHMC_contractive_hs} in finite dimensions. All that is left to do is to take the limit $N \to \infty$.
	\begin{proof}[Proof of \thmref{thm:exact_pHMC_contractive_hs}]
		It is similar to the proof of \thmref{thm:contraction_hs} but with references to \propref{prop:exact_dynamics_fd_converge} instead of \propref{prop:discrete_dynamics_inequalities} for pointwise convergence and dominated convergence.
	\end{proof}
	
	\begin{proof}[Proof of \thmref{thm:contraction_hs}]
		For $q, v \in \mathcal{H}^l$ we denote by $q^N, v^N$ the projections to $\mathcal{H}^N \approx \bb{R}^N$. After applying one iteration of the pHMC algorithm in $\bb{R}^N$ we get the output ${q^N}', {v^N}'$. Let $x, y \in \mathcal{H}^l$. By \thmref{thm:contraction_Rd} we know that
		\begin{equation}
			\label{equ:splitting_up_fdd_contraction}
			\bb{E}\left[\norm{{x^N}' - {y^N}'}_l \right] \le (1 - \epsilon)\norm{x^N - y^N}_l \le (1 - \epsilon)\norm{x - y}_l
		\end{equation}
		for $\epsilon = \frac{1}{27}\zeta T^2$. We want to use dominated convergence on the left hand side. We now use the same uniform random variable $U \sim \mbox{Unif}(\intcc{0, 1})$ in the acceptance probability for the algorithm on the Hilbert space itself and all the discretizations. We also sample $v$ directly on the Hilbert space $\mathcal{H}$ and use the projection $v^N$ of $v$ to $\mathcal{H}^N$ as the input for the finite dimensional algorithms. To be specific, one step of the finite dimensional algorithm consists of first projecting $v \sim \mathcal{N}(0, C)$ to $\mathcal{H}^N$ and calling the projection $v^N$. We then propose $({q^N}', {v^N}') = \psi^{(T)}_{h, N}$. The move is accepted if $U \le \alpha_N(q^N, v^N)$. We define
		\begin{align*}
			&A^N(x) = \{U \le \alpha_N(x, v)\} \text{ and } A(x) = \{U \le \alpha(x, v)\}.
		\end{align*}
		We now split up the expectation on the r.h.s of (\ref{equ:splitting_up_fdd_contraction}):
		\begin{equation}
			\begin{aligned}
				\bb{E}\left[\norm{{x^N}' - {y^N}'}_l \right] =&	\bb{E}\left[R'(x_N, y_N)\right] \\
				=& \bb{E}\left[\norm{q_T^N(x, v) - q_T^N(y, v)}_l 1_{A^N(x) \cap A^N(y)}\right] \\
				&+ \bb{E}\left[\norm{q_T^N(x, v) - y^N}_l 1_{A^N(x) \cap A^N(y)^C}\right] \\
				&+ \bb{E}\left[\norm{q_T^N(y, v) - x^N}_l 1_{A^N(x)^C \cap A^N(y)}\right] \\
				&+ \bb{E}\left[\norm{x^N - y^N}_l 1_{A^N(x)^C \cap A^N(y)^C}\right].
			\end{aligned}
		\end{equation}
		Since, for every $v$, the acceptance probabilities $\alpha^N(x,v)$ and $\alpha^N(y, v)$ converge to $\alpha(x, v)$ and $\alpha(y, v)$ respectively and the random variable $U$ used in the acceptance events $A^N(q)$ and $A(q)$ is the same we conclude that $1_{A^N(q)} \to 1_{A(q)}$ almost surely. By \propref{prop:discrete_dynamics_inequalities} we know that $x^N \to x, y^N \to y$. \lemref{lemma:convergence_acceptance_prob} on the other hand tells us that $q_T^N(x, v) \to q_T(x, v)$ and $q_T^N(y, v) \to q_T(y, v)$ for $\Pi_0$-almost every $(x, v)$ and $(y, v)$ respectively. Therefore all integrands above converge almost surely to their corresponding infinite dimensional counterpart, i.e.
		\begin{equation*}
			R'(x_N, y_N) \to R'(x, y)
		\end{equation*}
		almost surely. To use dominated convergence we note that by \propref{prop:discrete_dynamics_inequalities} we know that each of the terms $\norm{x^N}_l, \norm{y^N}_l, \norm{q_T^N(x, v)}_l, \norm{q_T^N(y,v)}_l$ can be bounded by $K(1 + \norm{x}_l + \norm{y}_l + \norm{v}_l) < \infty$. Therefore we can use triangle inequalities on each of the norm terms, e.g. $\norm{x^N - y^N}_l \le \norm{x^N}_l + \norm{y^N}_l$, and then bound all of them by $K(1 + \norm{x}_l + \norm{y}_l + \norm{v}_l)$ which is integrable since $\norm{v}_l$ is integrable by Fernique's theorem (see \citet{hairer_spde}). The indicator function terms can be bounded by $1$. Therefore we can use dominated convergence and conclude the proof.
	\end{proof}
	\section{Numerical Experiments}
	\label{sec:numerical_experiments}
	We now do some numerical experiments to illustrate the proved convergence bounds.
	We sample $\exp(-\Phi(q))\dif\pi_0$ where $\pi_0$ is the Brownian Bridge measure on $L^2([0,1])$. In the first experiment, $\Phi$ is defined as $\Phi(q) = \int_0^1 q(s) \dif s$. In the second experiment we set $\Phi(q) = (\int_0^1 q(s)^2 \dif s - 1)^2$. We use the orthonormal basis $\{\sqrt{2}\sin(n\pi x)\}_{i=1}^\infty$ of $L^2([0,1])$ (which we got from Karhunen-Lo\`{e}ve expansion of $\pi_0$) and expand $q$ as 
	$$q = \sum_{i=1}^N q_i \sqrt{2}\sin(n\pi x).$$ 
	We represent $q$ through the $q_i$ in the numerical simulation. This makes $C$ a diagonal matrix. $N$ is chosen to be $5000$. We use a stepsize of $h=0.2$ and make $13$ steps in each iteration, i.e. $T = 12*0.2 = 2.4$. For both experiments, we plot the first few iterates. We also plot the iterations of pHMC against the $L^2$-distance of the two samples. One iteration means taking $12$ steps of stepsize $h=0.2$ and then either accepting or rejecting the new sample. The two copies are coupled in the previously discussed way, by using the same velocity $v$ and the same uniform random variable $U \sim \text{Unif}(0,1)$ in the acceptance-rejection step. 
	
	We plot the initial state of the two coupled pHMC algorithms in the first picture of Fig. \ref{fig:linearcouplingsamples}. The second and third picture are the first iterates that the pHMC algorithms produced for the potential $\Phi(q) = \int_0^1 q(s) \dif s$. The coupled copies of pHMC for the potential $\Phi(q) = (\int_0^1 q(s)^2 \dif s - 1)^2$ got the same initial states but of course produced a different iterate 2 and iterate 3. 
	
	In Fig. \ref{fig:couplingdistancecomparison} we plot the $L^2$ distance of the iterates of the coupled algorithms. We see that they behave very similar for both potentials, although in the case $\Phi(q) = (\int_0^1 q(s)^2 \dif s - 1)^2$, there are a few points in the algorithm where one or both of the iterates is rejected and the distance does not decrease. For both potentials, around the $130$th iteration the two copies are so close to each other that they have the same numerical representation due to rounding happening in the computer. Shortly before the convergence happens, the distance decrease stops being as smooth as before. In that regime this effect is probably due to rounding errors.
	\begin{figure}[!t]
		\centering
		\includegraphics[width=\linewidth]{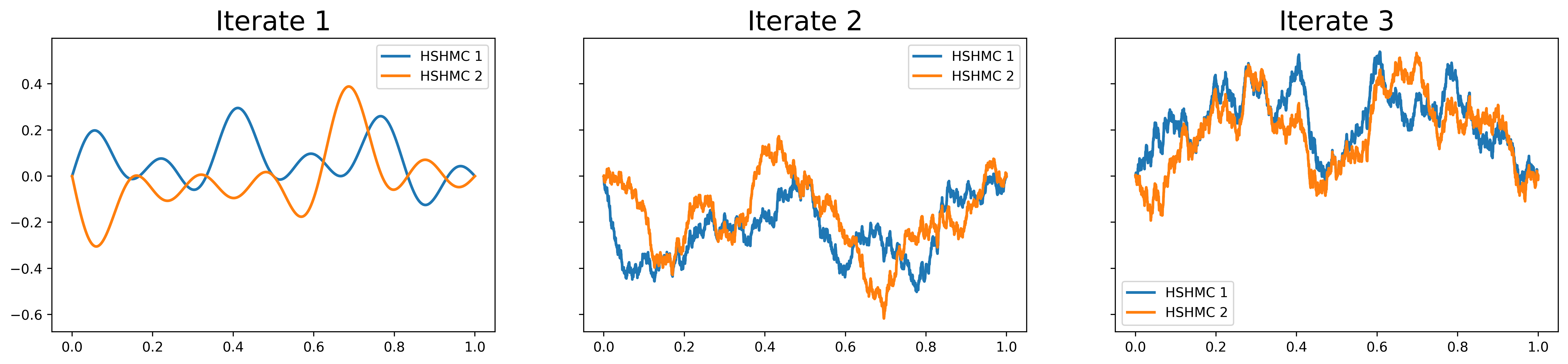}
		\caption{The samples produced by the first 3 iterations of two coupled copies of pHMC for the probability measure proportional to $\exp(-\Phi(q))\dif \pi_0$ with $\Phi(q) = \int_0^1 q(s) \dif s$ and $\pi_0$ the Brownian Bridge measure on $L^2([0,1])$. The first iterate was input to the algorithm as initial state.}
		\label{fig:linearcouplingsamples}
	\end{figure}

	\begin{figure}
		\centering
		\includegraphics[width=\linewidth]{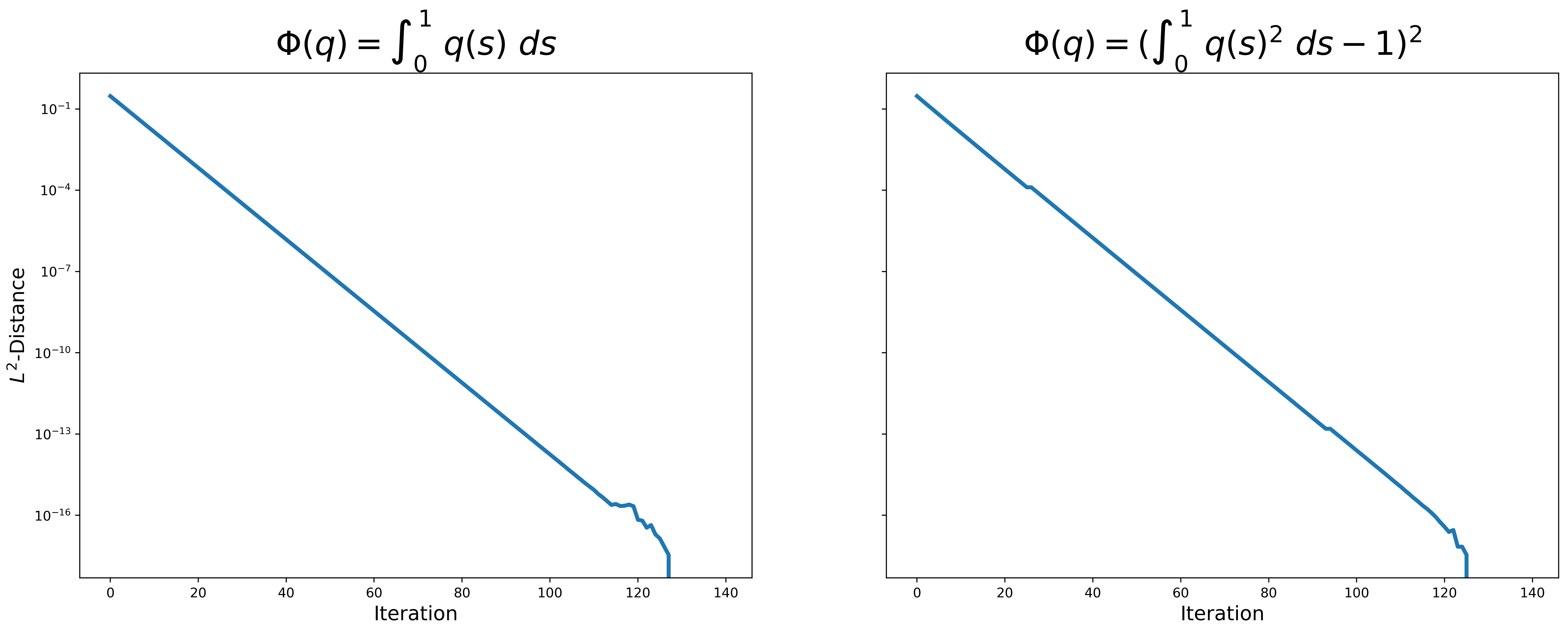}
		\caption{The $L^2$-distance of the iterates of two coupled copies of pHMC which are sampling the probability measure proportional to $\exp(-\Phi(q))\dif \pi_0$, where $\pi_0$ is the Brownian Bridge measure on $L^2([0,1])$. The qualitative and quantitative behaviour is similar.}
		\label{fig:couplingdistancecomparison}
	\end{figure}
	
	To check how stable the algorithm is under choice of the basis we use to represent $q$, we made another experiment. This time we represented $q$ as $\{q(x_1), q(x_2), \ldots, q(x_N)\}$, where $x_1, \ldots, x_N$ are $N = 5000$ evenly spaced points in $[0, 1]$. The covariance matrix is not diagonal any more, it has entries $\min(s, t) - st$. The initial convergence behaviour is nearly the same, but when the two iterates are very close it takes longer for them to be rounded to the same numerical representation. The $L^2$-distances are plotted in Fig. \ref{fig:couplingdistancecomparisonpointevalulations}.
	
	\begin{figure}
		\centering
		\includegraphics[width=\linewidth]{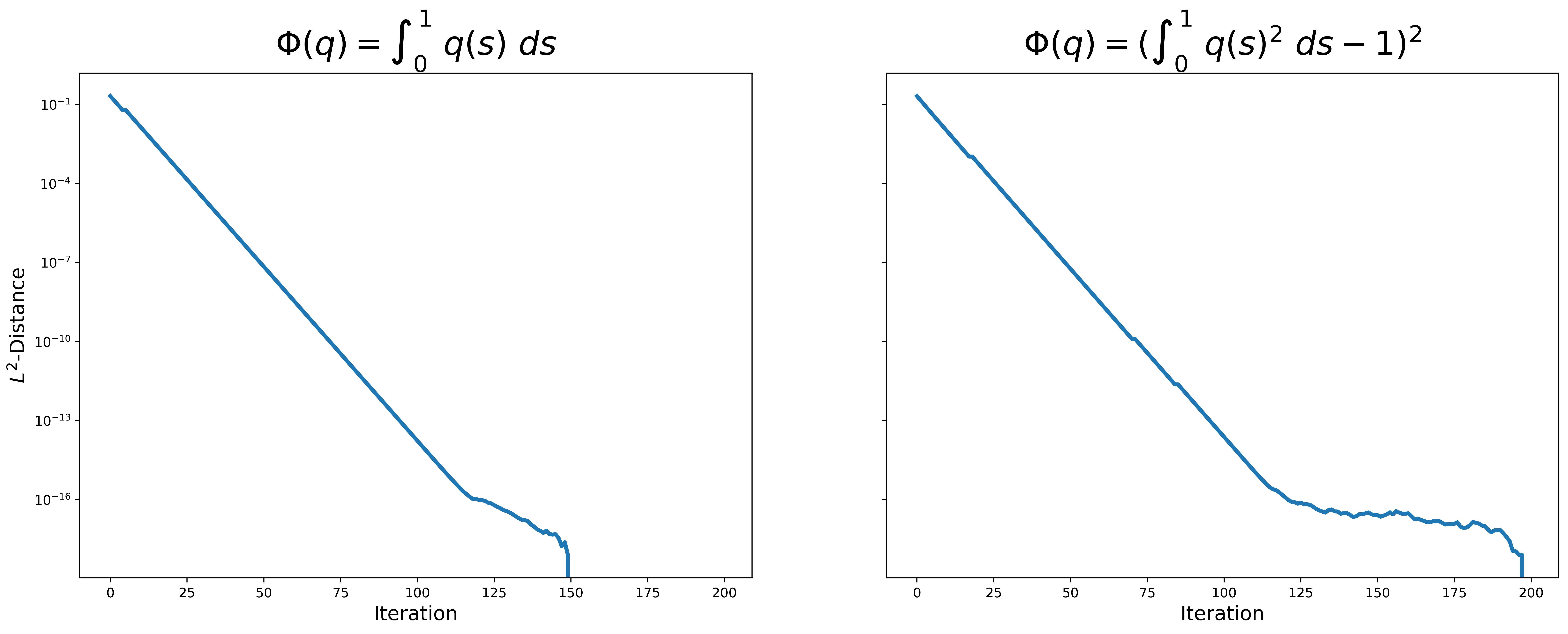}
		\caption{The same experiment as in Fig. \ref{fig:couplingdistancecomparison}, but numerically representing $q$ through point evaluations $\{q(x_1), \ldots, q(x_{5000})\}$ instead of the Karhunen-Lo\`{e}ve basis.}
		\label{fig:couplingdistancecomparisonpointevalulations}
	\end{figure}
	
	In our last experiment we took a double-well potential 
	\begin{equation}
		\label{equ:phi_double_well}
		\Phi(q) = \frac{\gamma}{2} \int_0^1 (q(s) - \frac{1}{2})^2(q(s) + \frac{1}{2})^2 \dif s = \frac{\gamma}{2} \int_0^1 (q(s)^2 - \frac{1}{4})^2 \dif s
	\end{equation}
	and started the pHMC algorithm near the two different minima of $\Phi$, which are the constant functions $\frac{1}{4}$ and $-\frac{1}{4}$. You see the initial iterates for $\gamma = 20$ in Fig. \ref{fig:doublewellgamma20couplingsamples}. You see the $L^2$-distance of the iterates for different values of $\gamma$ in Fig. \ref{fig:distancedoublewell}. Note that for this value of $\gamma$ pHMC has no problem transitioning between the two minima of $\Phi$, it does so directly in the first and second iteration. For $\gamma \le 20$ we observed that the iterates converged to each other in all of our experiments. For values bigger than that, it only sometimes seemed to converge. Using $\gamma$-values bigger than $40$ we never observed convergence. This does not contradict our results as $\Phi$ of course does not satisfy our assumptions. Especially it is not convex. A different coupling technique, as done in \citet{bou2018coupling} and \citet{bou2020two} could work better in these cases. In these experiments we again represented $q$ through its point evaluations on an evenly spaced grid of $N = 5000$ points.  
	Finally we plot the average distance over $20$ runs of the algorithm for different potentials in Fig. \ref{fig:4thordercouplingdistanceaverage}.
	\begin{figure}
		\centering
		\includegraphics[width=1\linewidth]{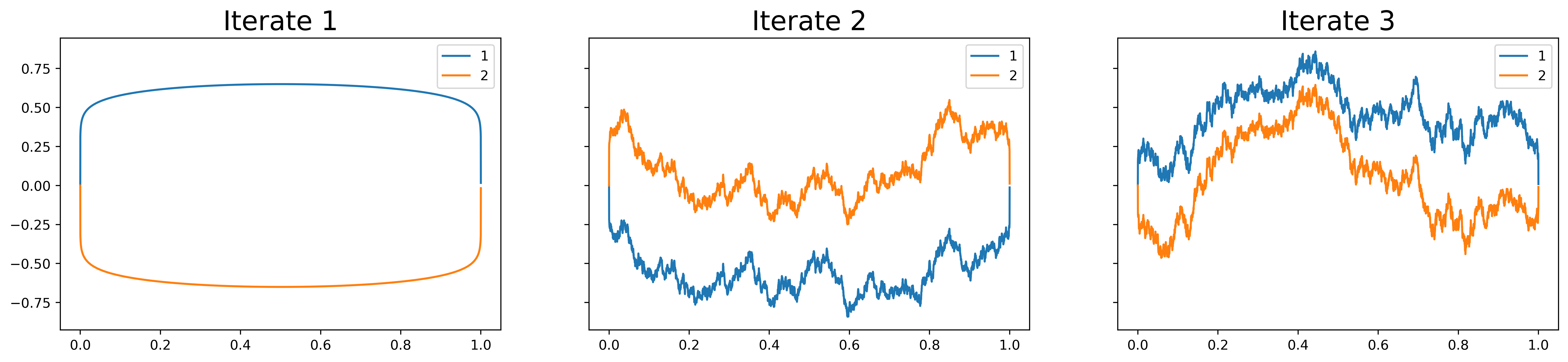}
		\caption{The initial iterates for two coupled pHMC algorithms sampling a measure proportional to $\exp(-\Phi(q))\dif \pi_0(q)$, where $\pi_0$ is the Brownian Bridge measure and $\Phi(q)$ is the double well potential $\frac{\gamma}{2} \int_0^1 (q(s)^2 - \frac{1}{4})^2 \dif s$ and $\gamma = 20$.}
		\label{fig:doublewellgamma20couplingsamples}
	\end{figure}
	\begin{figure}
		\centering
		\includegraphics[width=1\linewidth]{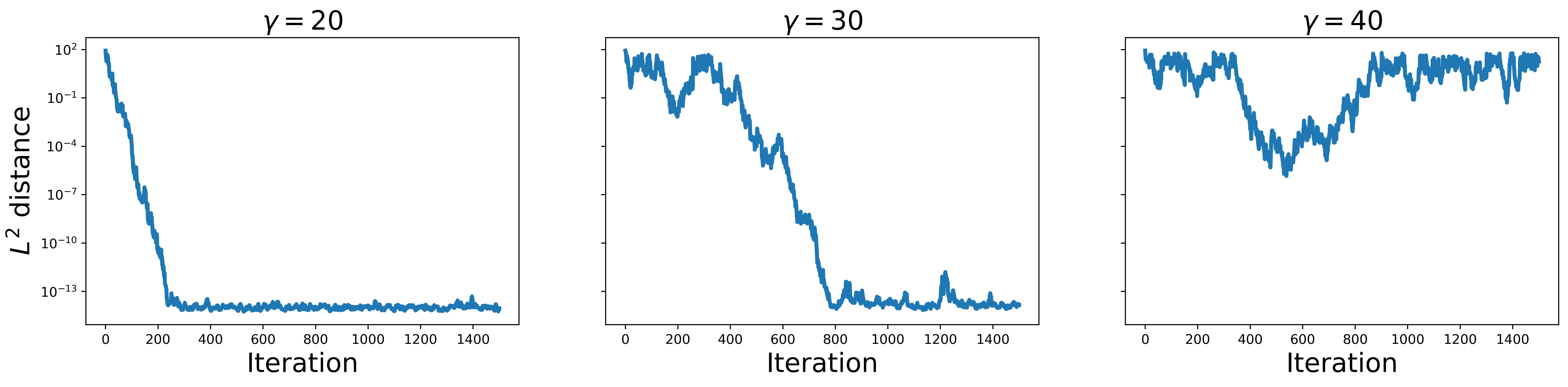}
		\caption{The $L^2$ distance of the coupled iterates of the coupled pHMC algorithms sampling the same measure as in Fig. \ref{fig:doublewellgamma20couplingsamples} for different values of $\gamma$.}
		\label{fig:distancedoublewell}
	\end{figure}
	\begin{figure}
		\centering
		\includegraphics[width=\linewidth]{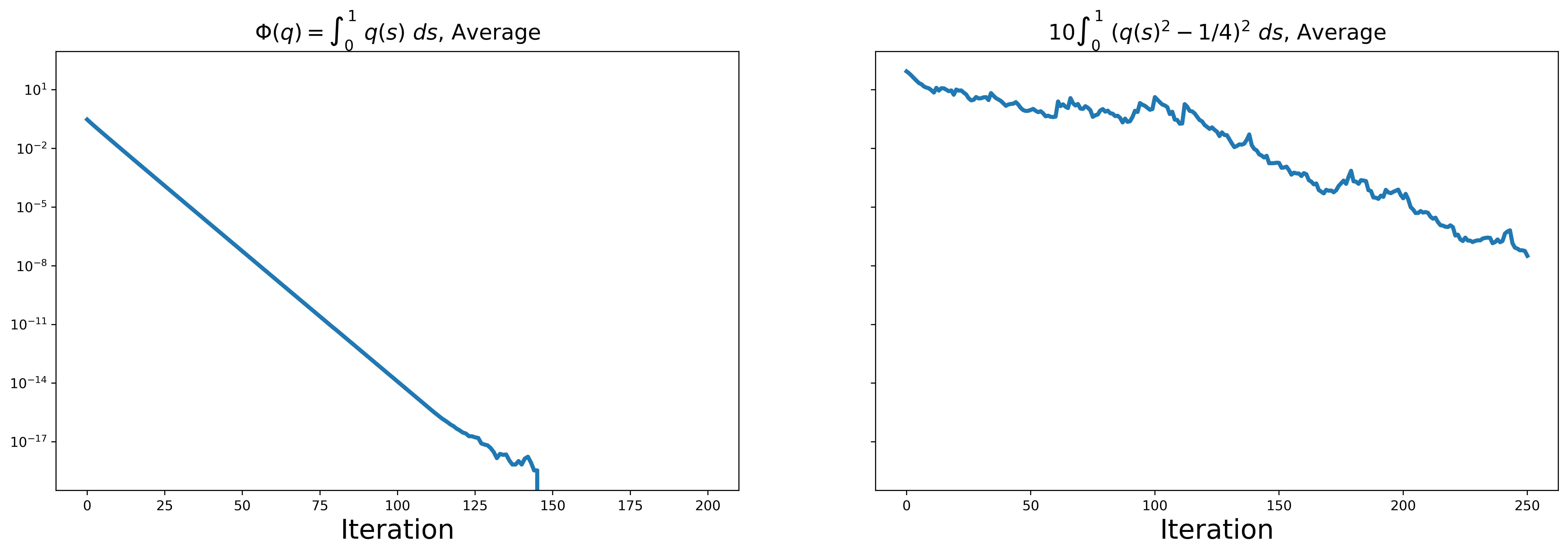}
		\caption{The same setups as in Fig. \ref{fig:couplingdistancecomparison} and Fig. \ref{fig:distancedoublewell} for $\gamma = 20$. This time we plot the average distance over 20 runs of the algorithm.}
		\label{fig:4thordercouplingdistanceaverage}
	\end{figure}
	\subsection{Conclusion}
	We have proven some quantitative bounds on the contraction rate of exact pHMC and showed that adjusted pHMC also contracts under restrictions on the size of the norm of the position variable $\norm{q}_l$. These results hold on the Hilbert space itself and do not depend on the embedding dimension. Further work could transfer the approaches of \citet{bou2018coupling} and \citet{bou2020two} to work for adjusted pHMC and multimodal potentials. Another interesting route is to understand how unadjusted pHMC behaves in infinite dimensions, analogous to the work \citet{bou2020convergence}. 
	
	\subsection{Acknowledgement}
I want to thank Andreas Eberle for choosing this exciting topic for my master thesis and supporting me during the writing. Furthermore, I want to thank Nawaf Bou-Rabee and Katharina Schuh for supporting me during the preparation of this publication and helping me in understanding new advances in this area. I also want to thank the SFB1294 for providing a productive working environment.

	\bibliography{paper}
\end{document}